\documentclass[a4paper,11pt]{article}
\usepackage[utf8]{inputenc}
\usepackage{amsmath}
\usepackage{amsfonts}
\usepackage{amssymb}
\usepackage{amsthm}
\usepackage[shortlabels]{enumitem}
\usepackage{tikz}
\usetikzlibrary{calc, decorations.pathreplacing, calligraphy, shapes.geometric}
\usepackage{xfp}
\usepackage{hyperref}
\usepackage{bm}
\usepackage[all]{hypcap} 
\usepackage{doi}

\newtheorem{theorem}{Theorem}[section]

\newtheorem{corollary}[theorem]{Corollary}
\newtheorem{proposition}[theorem]{Proposition}
\newtheorem{observation}[theorem]{Observation}
\newtheorem{lemma}[theorem]{Lemma}

\theoremstyle{definition}
\newtheorem{definition}[theorem]{Definition}

\newcommand{\w}{2.5pt}

\date{}

\begin{document}
\title{
Cubic graphs with edges in exactly one perfect matching}
\author{J. Goedgebeur$^{1,2}$\thanks{Supported by Internal Funds of KU Leuven and a grant of the Research Foundation Flanders (FWO) with grant number G0AGX24N.}, D. Mattiolo$^1$\thanks{Supported by a Postdoctoral Fellowship of the Research Foundation Flanders (FWO) with grant number 1268323N.}, G. Mazzuoccolo$^3$, \\ J. Renders$^{1*}$, I. H. Wolf$^4$\thanks{Funded by Deutsche Forschungsgemeinschaft (DFG) - 445863039.} \\
	\footnotesize $^1$ Department of Computer Science, KU Leuven Kulak, 8500 Kortrijk, Belgium.\\
 	\footnotesize $^2$ Department of Applied Mathematics, Computer Science and Statistics, \\ \footnotesize Ghent University, 9000 Ghent, Belgium.\\
  \footnotesize
 $^3$ Dipartimento di Scienze Fisiche, Informatiche e Matematiche,\\
 \footnotesize Universit\`a di Modena e Reggio Emilia, 41125 Modena, Italy.
\\
    \footnotesize $^4$ Department of Mathematics, Paderborn University, 33098 Paderborn,
		Germany.}
    \maketitle

\begin{abstract} 
       Petersen's seminal work in 1891 asserts that the edge-set of a cubic graph can be covered by distinct perfect matchings if and only if it is bridgeless. Actually, it is known that for a very large fraction of bridgeless cubic graphs, every edge belongs to at least two distinct perfect matchings. In this paper, we study the class of non-double covered cubic graphs, i.e.\ graphs having an edge, called lonely edge, which belongs to exactly one perfect matching. First of all, we provide a reduction of the problem to the subclass $\cal U$ of $3$-connected cubic graphs.  Then, we furnish an inductive characterization of $\cal U$ and we study properties related to the count of lonely edges. In particular, denoting by $\mathcal{U}_k$ the subclass of graphs of $\cal U$ with exactly $k$ lonely edges, we prove that $\mathcal{U}_k$ is empty for $k>6$, and we present a complete characterization for $3 \leq k \leq 6$. The paper concludes with some insights on ${\cal U}_1$ and ${\cal U}_2$.

    \medskip
    

    \medskip


\end{abstract}

\section{Introduction}

In graph theory, perfect matchings stand as one of the most captivating concepts. They represent a set of edges where every vertex is incident to exactly one of these edges. The study of perfect matchings in regular graphs has been a subject of extensive investigation (see for example~\cite{KKN06},\cite{KMZ23},\cite{MMSW23},\cite{MM20},\cite{MS21},\cite{Ma11},\cite{Ma13}), in particular when applied to cubic graphs.  
Already in 1891, Petersen~\cite{Pe91} proved that every bridgeless cubic graph $G$ admits a perfect matching. An extension of Petersen's theorem due to Schönberger~\cite{Sc34} implies that every edge of a bridgeless cubic graph belongs to at least one perfect matching of $G$. In other words, the edge-set of a cubic graph can be covered by using a set of distinct perfect matchings if and only if $G$ is bridgeless. 
In~\cite{EKKKN}, the number of perfect matchings in a bridgeless cubic graph is proved to grow exponentially with respect to the order of the graph. Actually, for a large fraction of bridgeless cubic graphs every edge belongs to more than one perfect matching.  

We say that a graph $G$ is \emph{matching double covered} if every edge $e\in E(G)$ belongs to at least two distinct perfect matchings.

In this paper, we investigate the class of cubic graphs which are {\it not} matching double covered. In other words, a bridgeless cubic graph is not matching double covered if it contains at least one edge $e$ belonging to exactly one perfect matching. From now on, we call such an edge $e$ a \emph{lonely edge}. The number of lonely edges of a graph $G$ is denoted by $l(G)$. 

In Section~\ref{sec:2ec}, we show how we can reduce the general description of non-double covered cubic graphs to the subclass, denoted by $\cal U$, of $3$-connected ones. 
Since it was already proved in~\cite{EKSS10} (see Lemma 19) that a 3-connected cubic graph  which is not a Klee graph (see Definition~\ref{def:Klee}) is matching double covered, this permits to focus our attention on this special class of cubic graphs. In particular, note that cyclically $4$-edge-connected cubic graphs are matching double covered.

In Section~\ref{sec:3ec}, we present an inductive description of $\cal U$ and discuss some properties related to $l(G)$. In particular, denoting by $\mathcal{U}_k$  the set of graphs $G \in \mathcal{U}$ with $l(G)=k$, we show that $\mathcal{U}_k$ is empty for all $k>6$ and we offer a complete characterization for all $3 \leq k \leq 6$, see Corollary \ref{cor:5 and 6 are finite} and Theorem \ref{thm:Characterization_U4} and \ref{thm:Characterization_U3}. 
Finally, we conclude the paper with some insights on ${\cal U}_1$ and ${\cal U}_2$, see Theorem \ref{thm:U1_U2}.

\section{Reduction to \texorpdfstring{$\bm3$}{3}-connected cubic graphs}\label{sec:2ec}

In this section we show that the study of non-double covered cubic graphs can be reduced to $3$-edge-connected non-double covered cubic graphs.

Let $G_1$ and $G_2$ be two bridgeless cubic graphs (note that in this section graphs might have parallel edges) and let $x_iy_i\in E(G_i),$ for $i\in\{1,2\}.$ A \emph{$2$-cut-connection} between $G_1$ and $G_2$ is the following operation: for $i\in\{1,2\}$, remove $x_iy_i$ from $G_i$ and add the edges $x_1x_2$ and $y_1y_2$. The resulting graph, denoted by $(G_1,x_1y_1) \oplus (G_2,x_2y_2)$ is bridgeless, cubic and $\{x_1x_2,y_1y_2\}$ is a $2$-edge-cut. 
The inverse operation is called a \emph{$2$-cut-reduction}. More precisely, the graph $G-F+x_1y_1+x_2y_2$ is obtained via a $2$-cut-reduction along $F$, where $G$ is a cubic graph with a $2$-edge-cut $F=\{x_1x_2,y_1y_2\}\subseteq E(G)$. Its connected components are denoted by $G_1$ and $G_2$.
In all cases considered we use the same labels for vertices and edges in $G$ and the corresponding vertices and edges in $G_i$.

\begin{observation}\label{obs:pm_2-cut-connection}
Let $G$ be a bridgeless cubic graph with a $2$-edge-cut $\{x_1x_2,y_1y_2\}$, and let $G_1$ and $G_2$ be two bridgeless cubic graphs such that $G = (G_1,x_1y_1) \oplus (G_2,x_2y_2)$. Then, $M\subseteq E(G)$ is a perfect matching of $G$ if and only if there are two perfect matchings $M_1$ and $M_2$ of $G_1$ and $G_2$ respectively, such that
    \begin{itemize}
        \item either $M=M_1\cup M_2$;
        \item or $x_iy_i\in M_i$ for $i\in\{1,2\}$ and $M=(M_1\cup M_2 \setminus \{x_1y_1,x_2y_2\})\cup \{x_1x_2,y_1y_2\}$. 
    \end{itemize}
\end{observation}

\begin{lemma}\label{lem:lonely_edges_in_2_cut_connection}
    Let $G$ be a bridgeless cubic graph with a $2$-edge-cut $\{x_1x_2,y_1y_2\}$, and let $G_1$ and $G_2$ be two bridgeless cubic graphs such that $G = (G_1,x_1y_1) \oplus (G_2,x_2y_2)$. The following statements hold.

    \begin{enumerate}[i)]
        \item Let $x_1y_1$ be lonely in $G_1$,
        and let $e\ne x_2y_2$ be a lonely edge of $G_2$. Then, $e$ is lonely in $G$ if and only if $x_2y_2$ belongs to the only perfect matching containing $e$ in $G_2$.
        \item Let $x_1y_1$ and $x_2y_2$ be lonely edges of $G_1$ and $G_2$, respectively. Then, $x_1x_2$ and $y_1y_2$ are both lonely in $G$;
        \item Let $x_1y_1$ be not lonely in $G_1$. Then, there is no edge of $G_2$ that is lonely in $G$.\\
        Moreover, $x_1x_2$ and $y_1y_2$ are also not lonely in $G$;
        \item For both $i\in\{1,2\}$, if an edge different from $x_iy_i$ is not lonely in $G_i$, then it is not lonely in $G$.
    \end{enumerate}
\end{lemma}

\begin{proof}

    Proof of point $i)$. Let $e\in E(G_2)\setminus \{x_2y_2\}$ be lonely in $G$. 
    Since $G_1$ is bridgeless, it has at least two perfect matchings, say $M_1$ and $M_2$, not containing $x_1y_1$, if the only perfect matching $N_1$ in $G_2$ which contains $e$ does not contain $x_2y_2$, then $M_1 \cup N_1$ and $M_2 \cup N_1$ are two distinct perfect matchings of $G$ containing $e$, a contradiction. 

    On the other hand, assume that $x_2y_2$ belongs to the only perfect matching  of $G_2$ containing $e.$ Since $x_1y_1$ is lonely, by Observation~\ref{obs:pm_2-cut-connection} it follows that $e$ is lonely in $G$.


    Point $ii)$ follows easily from Observation~\ref{obs:pm_2-cut-connection}. 

    Proof of point $iii)$. Since $x_1y_1$ is not lonely in $G_1$, there exist two different perfect matchings of $G_1$ containing it, say  $M_1$ and $M_2$. Moreover, since $G_1$ is bridgeless there exists at least two perfect matchings of $G_1$ not containing $x_1y_1,$ say $M_3$ and $M_4.$ 
    Then, by Observation~\ref{obs:pm_2-cut-connection} every perfect matching of $G_2$ can be combined with either $M_1$ and $M_2$, or $M_3$ and $M_4$, to obtain two perfect matchings in $G$ containing any prescribed edge in $E(G_2) \setminus \{x_2y_2\}$.

    Proof of point $iv)$ follows similarly as above.
    \end{proof}

\begin{theorem}\label{thm:inf_family_k_lonely_edges}
    For every positive integer $k$ there is a bridgeless cubic graph with exactly $k$ lonely edges.
\end{theorem}

\begin{proof}
   We prove a slightly stronger version of the statement by induction. Indeed, we prove the existence of an example with the additional property that all $k$ edges belong to the same perfect matching.    
    For $k=1$, an example is shown in Figure~\ref{fig:1lonelyEdge} and was found by computer, see Section \ref{sec:3ec}.

    Let $G_1$ be a bridgeless cubic graph with a unique lonely edge $x_1y_1$, then there is exactly one perfect matching containing $x_1y_1$. Now, let $k\ge2$ and assume by the inductive hypothesis that there is a graph $G_{k-1}$ with exactly $k-1$ lonely edges all belonging to the same perfect matching. Let $x_2y_2$ be one of these lonely edges. Then, by Lemma~\ref{lem:lonely_edges_in_2_cut_connection}, $(G_1,x_1y_1) \oplus (G_{k-1},x_2y_2)$ is a bridgeless cubic graph with exactly $k$ lonely edges: $x_1x_2, y_1y_2$ and the $k-2$ lonely edges in $G_{k-1}$ different from $x_2y_2$. Note that all such edges are contained in the same perfect matching. The statement then follows.
\end{proof}

By Lemma~\ref{lem:lonely_edges_in_2_cut_connection}, it is sufficient to study this problem in the $3$-connected case (and thus without parallel edges). Indeed, let $G$ be a bridgeless cubic graph with $2$-edge-cuts. We can successively apply the $2$-cut reduction operation until the graph is a disjoint union of cubic graphs $G_1, \ldots, G_k$: each $G_i$ is  either a graph with 2 vertices and 3 parallel edges or it is $3$-connected.

\section{\texorpdfstring{$\bm 3$}{3}-connected non-double covered cubic graphs}\label{sec:3ec}

In this section we present a complete characterization of $\mathcal{U}_k$ for $k\in\{3,\dots,6\}$ and we give some insights on $\mathcal{U}_1$ and $\mathcal{U}_2$.

Let $G$ be a cubic graph and $v\in V(G)$. Then, $G^v$ is the cubic graph obtained from $G$ by replacing the vertex $v$ with a triangle. See Figure~\ref{fig:blowup_example}.
More precisely, if $x_1,x_2$ and $x_3$ are the three vertices adjacent to $v$ in $G$ and $v_1,v_2,v_3$ are three new vertices, we have $$V(G^v)=(V(G) \setminus \{v\}) \cup \{v_1,v_2,v_3\},$$ 
$$E(G^v)=(E(G) \setminus \{vx_1,vx_2,vx_3\}) \cup \{v_1x_1,v_2x_2,v_3x_3\} \cup \{v_1v_2,v_2v_3,v_3v_1\}.$$ 

\begin{figure}[!htb]
    \newcommand{\size}{3cm} 
    \def\Arrow{\raisebox{12\height}{$\longrightarrow$\,\,}}
    \centering
    \tikzstyle{vertex} = [fill, black, circle, scale=0.5]
    \begin{tikzpicture}
        \node[vertex, label=above right:$v$] (0) at (0,0) {};
        \node[regular polygon, regular polygon sides=3, minimum size=\size] (T) at (0,0) {};
        \foreach \x in {1,2,...,3} {
            \node[vertex] (\x) at (T.corner \x) {};
            \draw[] (0) to (\x);
        }
        \node[label=$x_1$] at (1){};
        \node[label=below left:$x_2$] at (2){};
        \node[label=below right:$x_3$] at (3){};
    \end{tikzpicture}
    \Arrow
    \begin{tikzpicture}
        \node[regular polygon, regular polygon sides=3, minimum size=0.5*\size, draw] (T1) at (0,0) {};
        \foreach \x in {1,2,...,3} {
            \node[vertex] (\x) at (T1.corner \x) {};
        }
        \node[regular polygon, regular polygon sides=3, minimum size=\size] (T2) at (0,0) {};
        \foreach \x in {1,2,...,3} {
            \node[vertex] (\x') at (T2.corner \x) {};
            \draw (\x) to (\x');
        }

        \node[label={[label distance=-0.1cm]right:$v_1$}] at (1){};
        \node[label={[label distance=-0.25cm]above left:$v_2$}] at (2){};
        \node[label={[label distance=-0.25cm]above right:$v_3$}] at (3){};
        
        \node[label=$x_1$] at (1'){};
        \node[label=below left:$x_2$] at (2'){};
        \node[label=below right:$x_3$] at (3'){};
    \end{tikzpicture}
    \caption{An example of replacing $v$ with a triangle.}
    \label{fig:blowup_example}
\end{figure}
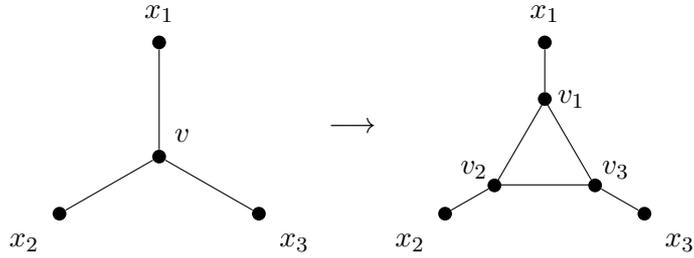

Along the entire paper we largely make use of this operation so we will now give some more details about the terminology we will use. 
We refer to the three edges $v_1v_2,v_2v_3$ and $v_3v_1$ of $G^v$ as the edges of the \emph{new triangle}, or more briefly as the \emph{new edges}. Each edge $v_ix_i$ will be said to be the edge \emph{opposite} to $v_jv_k$, for pairwise different $i,j,k \in \{1,2,3\}$. We refer to all other edges different from the three new edges as \emph{old edges} (thus including the three edges $v_1x_1,v_2x_2$ and $v_3x_3$). Moreover,  with a slight abuse of terminology, we use the same label for all old edges both when we consider them as edges of $G$ and of $G^v$. In particular, we identify an edge $vx$ in $G$ with the corresponding edge $v_ix$ in $G^v$.  
Finally, when we consider the reverse operation, we say that a graph is obtained by \emph{contracting} a triangle into a vertex. 

\begin{definition}\label{def:Klee}
    A graph $G$ is a \emph{Klee-graph} if $G = K_4$ or there exists a Klee-graph  $H$ and a vertex $w\in V(H)$ such that $G = H^w$. We denote the set of Klee-graphs by $\mathcal{K}$.
\end{definition}

The class of Klee-graphs is mainly known for its property of being the class of all planar cubic graphs admitting a unique $3$-edge-coloring. This property was conjectured by Fiorini and Wilson in~\cite{FW78} and then proved by Fowler in his PhD Thesis (see~\cite{F98}).
As already remarked, by Lemma~19 in~\cite{EKSS10}, we have that $\mathcal{U}\subset \mathcal{K}$. Note also that $\mathcal{U}\neq \mathcal{K}$, indeed there are Klee-graphs which are double covered, the smallest of which being the truncation of $K_4$ (see Figure~\ref{fig:truncK4}).

Let $G\in \mathcal{K}$. Assume that $H$ can be constructed from $G$ by a sequence of expansions of vertices into triangles. Then we say that $H$ is a \emph{descendant} of $G$. If $v\in V(G)$ then we call $G^v$ a \emph{child} of $G$.

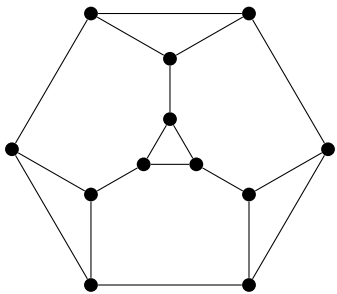
\begin{figure}[!htb]
    \centering
    \tikzstyle{vertex} = [fill, black, circle, scale=0.5]
    \begin{tikzpicture}
        \def\r{1.2}
        \def\rm{0.4}
        \def\R{1.2}
        \def\c{0}
        \foreach \x in {0,...,2}
            \node[vertex] (\fpeval{\x + 1}) at ({\rm*cos(90 + \x * 120)},{\rm*sin(90 + \x * 120)}) {};

        \foreach \x in {1,...,2}
            \node[vertex] (\fpeval{\x + 4}) at ({\r*cos(270 + \x * 120)},{\r*sin(270 + \x * 120) + \R}) {};
        \node[vertex] (4) at ({\c*\r*cos(270)},{\c*\r*sin(270) + \R}) {};

        \foreach \x in {1,...,2}
            \node[vertex] (\fpeval{\x + 7}) at ({\r*cos(2*120 - 90 + \x * 120) + \R*cos(2*120 + 90)},{\r*sin(2*120 - 90 + \x * 120) + \R*sin(2*120 + 90)}) {};
        \node[vertex] (7) at ({\c*\r*cos(2*120 - 90) + \R*cos(2*120 + 90)},{\c*\r*sin(2*120 - 90) + \R*sin(2*120 + 90)}) {};

        \foreach \x in {1,...,2}
            \node[vertex] (\fpeval{\x + 10}) at ({\r*cos(1*120 - 90 + \x * 120) + \R*cos(1*120 + 90)},{\r*sin(1*120 - 90 + \x * 120) + \R*sin(1*120 + 90)}) {};
        \node[vertex] (10) at ({\c*\r*cos(1*120 - 90) + \R*cos(1*120 + 90)},{\c*\r*sin(1*120 - 90) + \R*sin(1*120 + 90)}) {};

        \draw (1) to (2) to (3) to (1);
        \draw (4) to (5) to (6) to (4);
        \draw (7) to (8) to (9) to (7);
        \draw (10) to (11) to (12) to (10);

        \draw (1) to (4);
        \draw (2) to (10);
        \draw (3) to (7);
        \draw (5) to (9);
        \draw (6) to (11);
        \draw (8) to (12);
    \end{tikzpicture}
    \caption{The truncation of $K_4$ is the smallest Klee-graph which is double covered.}
    \label{fig:truncK4}
\end{figure}

\begin{lemma}
\label{lem:G^v KG then G KG}
Let $G$ be a cubic graph, $|V(G)|\geq 4$ and $v \in V(G)$. If $G^v$ is a Klee-graph, then $G$ is a Klee-graph.
\end{lemma}
\begin{proof}
We prove the statement by induction on $|V(G)|$. If $|V(G)|=4$, then $G^v$ is isomorphic to the prism and the statement holds. Next, assume $|V(G)|>4$ and that the lemma is true for all $G'$ with $|V(G')|<|V(G)|$. Let $T$ be the triangle in $G^v$ that is induced by the new edges. Since $G^v$ is a Klee-graph, there is a triangle $T'$ in $G^v$ whose contraction produces a Klee-graph $H$. If $T=T'$, then we are done. Hence, assume $T \neq T'$. Note that since every Klee-graph is 3-edge-connected and $G^v$ is not isomorphic to $K_4$, the triangles $T$ and $T'$ are vertex-disjoint. Let $H'$ be the graph obtained from $G^v$ by contracting $T'$ and then $T$. Since $H$ is a Klee-graph, the graph $H'$ is also a Klee-graph by the inductive hypothesis. As a consequence, $G$ is a Klee-graph, since it can be obtained from $H'$ by replacing a vertex by a triangle.
\end{proof}

\begin{observation}\label{obs:trNotContainingVertex}
    If $G$ is a Klee-graph, then for every vertex $v$ of $G$ there is a triangle not containing $v$.
\end{observation}

Let $G$ be a cubic graph and $v \in V(G)$. A  \emph{join rooted in} $v$ of $G$ is a spanning subgraph in which $v$ has degree 3 and every other vertex has degree 1.

\begin{lemma}\label{lem:join}
    Let $G$ be a Klee-graph. For every $v\in V(G)$ there exists a join rooted in $v$.
\end{lemma}
\begin{proof}
    We prove the statement by induction on the order $2n$ of $G$. If $n=2$, then $G = K_4$ and one can see easily that this holds.

    If $n>2$, assume the assertion holds for all Klee-graphs of order less than $2n$. We will show that the assertion holds for a graph $G$ of order $2n$. Let $v$ be an arbitrary vertex of $G$.
    
    By Observation~\ref{obs:trNotContainingVertex}, there is a triangle $T$ not containing $v$. Let $G'$ be the graph obtained by contracting $T$ to a vertex $x$. Note that $G'$ is a Klee-graph by Lemma~\ref{lem:G^v KG then G KG}. By inductive hypothesis $G'$ contains a join $F'$ rooted in $v$. It is easy to see that $F'$ can be extended to a join $F$ rooted in $v$ of $G$. 
\end{proof}

\begin{lemma}\label{lem:blowUpToTriangle}
    Let $G\in \mathcal{U}$, $v \in V(G)$, $e$ be an arbitrary old edge of  $G^v$ and let $f$ and $f'$ be a new edge and its opposite edge, respectively. Then the following holds:
\begin{itemize}
\item[i)] $e$ is lonely in $G^v$ if and only if $e$ is lonely in $G$ and there is no join rooted in $v$ of $G$ containing $e$. 
\item[ii)] $f'$ is not lonely in $G^v$.
\item[iii)] $f$ is lonely in $G^v$ if and only if $f'$ is lonely in $G$. 

\end{itemize}

\end{lemma}
\begin{proof}
We first prove $i)$. Every perfect matching in $G$ containing $e$ can be extended to a perfect matching of $G^v$ by adding a suitable new edge. Hence, if $e$ is not lonely in $G$, then $e$ is not lonely in $G^v$. Furthermore, a join rooted in $v$ in $G$ corresponds to a perfect matching in $G^v$. Hence, if $e$ is contained in a join rooted in $v$ of $G$, then $e$ is not lonely in $G^v$. On the other hand, if $e$ is not lonely in $G^v$, then there are two distinct perfect matchings in $G^v$ containing $e$. If one of them does not contain an edge of the new triangle, then it is a join rooted in $v$ in $G$ containing $e$. If both of them contain an edge of the new triangle, then we obtain two distinct perfect matchings in $G$ containing $e$, i.e.\ $e$ is not lonely in $G$.

Point $ii)$ follows by Point $i)$ and Lemma~\ref{lem:join}. 

Next, we prove $iii)$. If $f$ is not lonely in $G^v$, then there are two distinct perfect matchings in $G^v$ containing $f$. Note that both of them contain $f'$. Hence, by removing $f$ from both of them we obtain two distinct perfect matchings of $G$ containing $f'$, i.e.\ $f'$ is not lonely in $G$. On the other hand, if $f'$ is not lonely in $G$, then there are two distinct perfect matchings in $G$ containing $f'$. Adding $f$ to both of them yields two distinct perfect matchings of $G^v$ containing $f$, i.e.\ $f$ is not lonely in $G^v$. Hence, the equivalence is proved. 
\end{proof}

\begin{corollary}
\label{cor:formula special edges}
For every $G \in \mathcal{U}$, $v \in V(G)$,
\begin{equation*}
l(G^v)=l(G)-p_v,
\end{equation*}
where $p_v$ is the number of lonely edges of $G$ not incident with $v$, which are contained in a join rooted in $v$ of $G$.
\end{corollary}

\begin{proposition}\label{prop:inductiveDefinition}
    $\mathcal{U}$ can be defined inductively as follows: \begin{enumerate}[i)]
        \item\label{it:K4inU} $K_4\in\mathcal{U}$,
        \item\label{it:noPMWithTriangleEdges} If $G\in \mathcal{U}$, $e\in E(G)$ is lonely, and  $v\in V(G)$ such that it is either incident with $e$ or there is no join rooted in $v$ containing $e$, then $G^v\in \mathcal{U}$.
    \end{enumerate}
\end{proposition}
\begin{proof}
    By~\ref{it:K4inU}, $K_4$ belongs to $\mathcal{U}$. Given a graph $G\in \mathcal{U}$ and a vertex $w$ satisfying the assumption of~\ref{it:noPMWithTriangleEdges}, we get by Corollary~\ref{cor:formula special edges} that $G^w$ belongs to $\mathcal{U}$.

    Now let $G$ be an arbitrary graph in $\mathcal{U}$. We show that it can be obtained by replacing a vertex  with a triangle in a smaller graph. Let $e$ be a lonely edge of $G$. By Observation~\ref{obs:trNotContainingVertex} there exists some triangle $T$ in $G$ not containing $e$. Let $H$ be the graph obtained from $G$ by contracting this triangle to a vertex $w$. Note that $e$ is a lonely edge of $H$, since it being contained in two distinct perfect matchings would give two distinct perfect matchings in $G$ containing $e$. Thus, $H$ is in $\mathcal{U}$. Since $H$ has a perfect matching containing $e$, there is a perfect matching in $G$ that contains $e$ and exactly one edge of $T$. Since the edge set of a join rooted in $v$ in $H$ is a perfect matching in $G$ which contains no edge of $T$, there is no join rooted in $v$ in $H$ containing $e$. Hence, $G$ is obtained from the procedure in~\ref{it:noPMWithTriangleEdges}.
\end{proof}

Corollary~\ref{cor:formula special edges} and Proposition~\ref{prop:inductiveDefinition} assure that the maximum number of lonely edges in a $3$-connected cubic graph is six and that this maximum is attained by the smallest graph in $\cal U$, that is $K_4$.

Using a computer search we determined all smallest graphs in $\mathcal{U}_k$ for every $k \in \{0,\ldots ,6\}$, see Figures~\ref{fig:truncK4}--\ref{fig:6lonelyEdges}, 
respectively.  In all next figures the lonely edges are dashed.

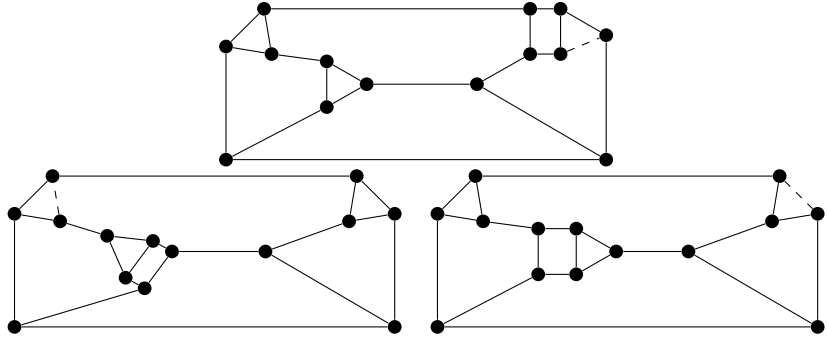
\begin{figure}[!htb]
    \newcommand{\x}{1}
    \centering
    \tikzstyle{vertex} = [fill, black, circle, scale=0.5]
    \begin{tikzpicture}[scale=\x]
        \node[vertex] (1) at (-0.4,0.4) {};
        \node[vertex] (1') at (-0.5,1) {};
        \node[vertex] (1") at (-1,0.5) {};

        \node[vertex] (2) at (-1,-1) {};
        \def\rm{0.35}
        \foreach \x in {0,...,2}
            \node[vertex] (\fpeval{\x + 3}) at ({\rm*cos(\x * 120)+0.5},{\rm*sin(\x * 120)}) {};
        \node[vertex] (6) at (2.3,0) {};
        
        \node[vertex] (7) at (3.4,0.4) {};
        \node[vertex] (7') at (3.4, 1) {};
        \node[vertex] (7") at (4, 0.65) {};
        
        \node[vertex] (8) at (4,-1) {};
        \path (7) ++(-0.4,0) node[vertex] (9) {};
        \path (7') ++(-0.4,0) node[vertex] (10) {};
        
        \draw (1) to (1');
        \draw (1') to (1");
        \draw (1) to (1");
        \draw (1) to (4);
        \draw (1') to (10);
        \draw (1") to (2);
        \draw (2) to (5);
        \draw (2) to (8);
        \draw (3) to (6);
        \draw (3) to (4) to (5) to (3);

        \draw (6) to (9);
        \draw (6) to (8);
        \draw (7") to (8);
        \draw (7) to (7');
        \draw (7') to (7");
        \draw[dashed] (7) to (7");
        \draw (7) to (9);
        \draw (7') to (10);
        \draw (9) to (10);
    \end{tikzpicture}\\
    \begin{tikzpicture}[scale=\x]
        \node[vertex] (1) at (-0.4,0.4) {};
        \node[vertex] (1') at (-0.5,1) {};
        \node[vertex] (1") at (-1,0.5) {};

        \node[vertex] (2) at (-1,-1) {};
        \def\rm{0.35}
        \foreach \x in {0,...,2}
            \node[vertex] (\fpeval{\x + 3}) at ({\rm*cos(23.5781785 + \x * 120)+0.5},{\rm*sin(23.5781785 + \x * 120)}) {};
        \node[vertex] (6) at (2.3,0) {};
        
        \node[vertex] (7) at (3.4,0.4) {};
        \node[vertex] (7') at (3.5, 1) {};
        \node[vertex] (7") at (4, 0.5) {};
        
        \node[vertex] (8) at (4,-1) {};
        
        \path (5) ++(0.25,-0.14) node[vertex] (9) {};
        \path (3) ++(0.25,-0.14) node[vertex] (10) {};
        
        \draw[dashed] (1) to (1');
        \draw (1') to (1");
        \draw (1) to (1");
        \draw (1) to (4);
        \draw (1') to (7');
        \draw (1") to (2);
        \draw (2) to (9);
        \draw (2) to (8);
        \draw (10) to (6);
        \draw (3) to (4) to (5) to (3);
        \draw (3) to (10);
        \draw (5) to (9);

        \draw (6) to (7);
        \draw (6) to (8);
        \draw (7") to (8);
        \draw (7) to (7');
        \draw (7') to (7");
        \draw (7) to (7");
        \draw (9) to (10);

    \end{tikzpicture}\quad
    \begin{tikzpicture}[scale=\x]
        \node[vertex] (1) at (-0.4,0.4) {};
        \node[vertex] (1') at (-0.5,1) {};
        \node[vertex] (1") at (-1,0.5) {};

        \node[vertex] (2) at (-1,-1) {};
        \def\rm{0.35}
        \foreach \x in {1,...,2}
            \node[vertex] (\fpeval{\x + 3}) at ({\rm*cos(\x * 120)+0.5},{\rm*sin(\x * 120)}) {};
        \foreach \x in {0,...,2}
            \node[vertex] (\fpeval{\x + 9}) at ({\rm*cos(\x * 120)+1},{\rm*sin(\x * 120)}) {};
        \node[vertex] (6) at (2.3,0) {};
        
        \node[vertex] (7) at (3.4,0.4) {};
        \node[vertex] (7') at (3.5, 1) {};
        \node[vertex] (7") at (4, 0.5) {};
        
        \node[vertex] (8) at (4,-1) {};
        
        \draw (1) to (1');
        \draw (1') to (1");
        \draw (1) to (1");
        \draw (1) to (4);
        \draw (1') to (7');
        \draw (1") to (2);
        \draw (2) to (5);
        \draw (2) to (8);
        \draw (9) to (6);
        \draw (10) to (4) to (5) to (11) to (9) to (10);
        \draw (10) to (11);

        \draw (6) to (7);
        \draw (6) to (8);
        \draw (7") to (8);
        \draw (7) to (7');
        \draw[dashed] (7') to (7");
        \draw (7) to (7");
    \end{tikzpicture}
    \caption{The three smallest graphs in $\mathcal{U}_1$.}
    \label{fig:1lonelyEdge}
\end{figure}

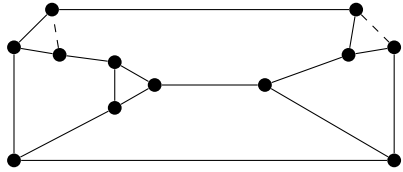
\begin{figure}[!htb]
    \centering
    \tikzstyle{vertex} = [fill, black, circle, scale=0.5]
    \begin{tikzpicture}
        \node[vertex] (1) at (-0.4,0.4) {};
        \node[vertex] (1') at (-0.5,1) {};
        \node[vertex] (1") at (-1,0.5) {};

        \node[vertex] (2) at (-1,-1) {};
        \def\rm{0.35}
        \foreach \x in {0,...,2}
            \node[vertex] (\fpeval{\x + 3}) at ({\rm*cos(\x * 120)+0.5},{\rm*sin(\x * 120)}) {};
        \node[vertex] (6) at (2.3,0) {};
        
        \node[vertex] (7) at (3.4,0.4) {};
        \node[vertex] (7') at (3.5, 1) {};
        \node[vertex] (7") at (4, 0.5) {};
        
        \node[vertex] (8) at (4,-1) {};
        
        \draw[dashed] (1) to (1');
        \draw (1') to (1");
        \draw (1) to (1");
        \draw (1) to (4);
        \draw (1') to (7');
        \draw (1") to (2);
        \draw (2) to (5);
        \draw (2) to (8);
        \draw (3) to (6);
        \draw (3) to (4) to (5) to (3);

        \draw (6) to (7);
        \draw (6) to (8);
        \draw (7") to (8);
        \draw (7) to (7');
        \draw[dashed] (7') to (7");
        \draw (7) to (7");
    \end{tikzpicture}
    \caption{The smallest graph in $\mathcal{U}_2$.}
    \label{fig:2lonelyEdges}
\end{figure}

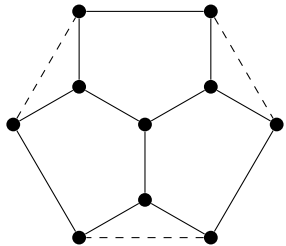
\begin{figure}[!htb]
    \centering
    \tikzstyle{vertex} = [fill, black, circle, scale=0.5]
    \begin{tikzpicture}
        \node[vertex] (1) at (0,0) {};
        \node[vertex] (2) at (0,-1) {};
        \node[vertex] (3) at (-0.86602540378,0.5) {};
        \node[vertex] (4) at (0.86602540378,0.5) {};
        \node[vertex] (5) at (0.86602540378,-1.5) {};
        \node[vertex] (6) at (-0.86602540378,-1.5) {};
        \node[vertex] (7) at (-1.73205080757,0) {};
        \node[vertex] (8) at (-0.86602540378,1.5) {};
        \node[vertex] (9) at (0.86602540378,1.5) {};
        \node[vertex] (10) at (1.73205080757,0) {};
        \draw (1) to (2);
        \draw (1) to (3);
        \draw (1) to (4);
        \draw (2) to (5);
        \draw (2) to (6);
        \draw (3) to (7);
        \draw (3) to (8);
        \draw (4) to (9);
        \draw (4) to (10);
        \draw[dashed] (5) to (6);
        \draw (5) to (10);
        \draw (6) to (7);
        \draw[dashed] (7) to (8);
        \draw (8) to (9);
        \draw[dashed] (9) to (10);
    \end{tikzpicture}\qquad
    \caption{The tricorn is the smallest graph in $\mathcal{U}_3$.
    }
    \label{fig:3lonelyEdges}
\end{figure}

\begin{figure}[!htb]
    \centering
    \tikzstyle{vertex} = [fill, black, circle, scale=0.5]
    \begin{tikzpicture}
        \node[vertex] (1) at (0.6,0.4) {};
        \node[vertex] (1') at (0.5,1) {};
        \node[vertex] (1") at (0,0.5) {};

        \node[vertex] (2) at (0,-1) {};
        \node[vertex] (3) at (1.2,0) {};
        \node[vertex] (4) at (2.8,0) {};
        \node[vertex] (5) at (4,1) {};
        \node[vertex] (6) at (3.4,-0.4) {};
        \node[vertex] (6') at (3.5, -1) {};
        \node[vertex] (6") at (4, -0.5) {};
        
        \draw[dashed] (1) to (1');
        \draw (1') to (1");
        \draw (1) to (1");
        \draw (1) to (3);
        \draw (1') to (5);
        \draw (1") to (2);
        \draw[dashed] (2) to (3);
        \draw (2) to (6');
        \draw (3) to (4);
        \draw[dashed] (4) to (5);
        \draw (4) to (6);
        \draw (5) to (6");
        \draw[dashed] (6) to (6');
        \draw (6') to (6");
        \draw (6") to (6);
    \end{tikzpicture}\qquad
    \begin{tikzpicture}
        \node[vertex] (1) at (0.6,0.4) {};
        \node[vertex] (1') at (0.5,1) {};
        \node[vertex] (1") at (0,0.5) {};

        \node[vertex] (2) at (0,-1) {};
        \node[vertex] (3) at (1.2,0) {};
        \node[vertex] (4) at (2.8,0) {};
        
        \node[vertex] (5) at (3.4,0.4) {};
        \node[vertex] (5') at (3.5, 1) {};
        \node[vertex] (5") at (4, 0.5) {};
        
        \node[vertex] (6) at (4,-1) {};
        
        \draw[dashed] (1) to (1') to (1");
        \draw (1) to (1");
        \draw (1) to (3);
        \draw (1') to (5');
        \draw (1") to (2);
        \draw (2) to (3);
        \draw (2) to (6);
        \draw (3) to (4);
        \draw (4) to (5);
        \draw (4) to (6);
        \draw (5") to (6);
        \draw[dashed] (5) to (5') to (5");
        \draw (5) to (5");
    \end{tikzpicture}
    \caption{The smallest graphs in $\mathcal{U}_4$.}
    \label{fig:4lonelyEdges}
\end{figure}
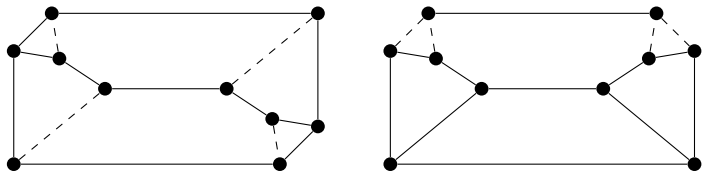

\begin{figure}[!htb]
    \centering
    \tikzstyle{vertex} = [fill, black, circle, scale=0.5]
    \begin{tikzpicture}
        \node[vertex] (1) at (0.6,0.4) {};
        \node[vertex] (1') at (0.5,1) {};
        \node[vertex] (1") at (0,0.5) {};

        \node[vertex] (2) at (0,-1) {};
        \node[vertex] (3) at (1.2,0) {};
        \node[vertex] (4) at (2.8,0) {};
        \node[vertex] (5) at (4,1) {};
        \node[vertex] (6) at (4,-1) {};
        
        \draw[dashed] (1) to (1') to (1");
        \draw (1) to (1");
        \draw (1) to (3);
        \draw (1') to (5);
        \draw (1") to (2);
        \draw[dashed] (2) to (3);
        \draw (2) to (6);
        \draw (3) to (4);
        \draw[dashed] (4) to (5);
        \draw (4) to (6);
        \draw[dashed] (5) to (6);
    \end{tikzpicture}
    \caption{The bicorn is the only graph in $\mathcal{U}_5$.}
    \label{fig:5lonelyEdges}
\end{figure}
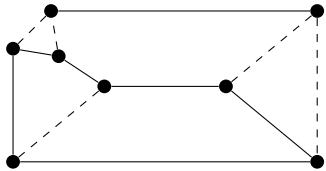

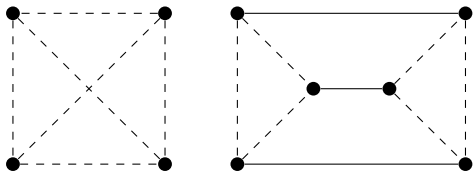
\begin{figure}[!htb]
    \centering
    \tikzstyle{vertex} = [fill, black, circle, scale=0.5]
    \begin{tikzpicture}
        \node[vertex] (1) at (1,1) {};
        \node[vertex] (2) at (-1,1) {};
        \node[vertex] (3) at (-1,-1) {};
        \node[vertex] (4) at (1,-1) {};
        \draw[dashed] (1) to (2);
        \draw[dashed] (1) to (3);
        \draw[dashed] (1) to (4);
        \draw[dashed] (2) to (3);
        \draw[dashed] (2) to (4);
        \draw[dashed] (3) to (4);
    \end{tikzpicture}\qquad
    \begin{tikzpicture}
        \node[vertex] (1) at (0,1) {};
        \node[vertex] (2) at (0,-1) {};
        \node[vertex] (3) at (1,0) {};
        \node[vertex] (4) at (2,0) {};
        \node[vertex] (5) at (3,1) {};
        \node[vertex] (6) at (3,-1) {};
        \draw[dashed] (1) to (2);
        \draw[dashed] (1) to (3);
        \draw (1) to (5);
        \draw[dashed] (2) to (3);
        \draw (2) to (6);
        \draw (3) to (4);
        \draw[dashed] (4) to (5);
        \draw[dashed] (4) to (6);
        \draw[dashed] (5) to (6);
    \end{tikzpicture}
    
    \caption{The only two graphs in $\mathcal{U}_6$.}
    \label{fig:6lonelyEdges}
\end{figure}

In particular, the next corollary easily follows by our previous considerations and computations.

\begin{corollary}\label{cor:5 and 6 are finite}
    The following statements hold.
    \begin{enumerate} 
        \item[i)] $l(G)\leq 6$ for every $G\in\mathcal{U}$;
        \item[ii)] $K_4$ and the prism are the only elements of $\mathcal{U}_6$ (see Figure~\ref{fig:6lonelyEdges});
        \item[iii)] The bicorn is the only element of $\mathcal{U}_5$ (see  Figure~\ref{fig:5lonelyEdges}).
    \end{enumerate}
\end{corollary}


We computed the smallest examples as follows. First, we generated $3$-connected cubic graphs using \texttt{snarkhunter}~\cite{BGM11}. Then, we computed for each graph the number of lonely edges by generating all of their perfect matchings and counting which edges appear in precisely one of them. The implementation of this algorithm can be found on GitHub~\cite{GMMRW23}. It also contains auxiliary methods for outputting all children of an input graph $G$ and for outputting precisely those children $G^v$ for which $l(G^v) = l(G)$. 

To verify the correctness of these implementations, we performed various tests. For filtering the $3$-connected graphs containing at least one lonely edge, we created two independent implementations. The results from both implementations are in complete agreement, which we tested up to order $24$. Moreover, using the program which outputs children of a given input graph, we started from $K_4$ and generated its children. We stored those that were non-double covered and repeated this process. The counts were again in agreement for every considered order in every class $\mathcal{U}_k$. We checked this up to order $30$.

The smallest examples of every class $\mathcal{U}_k$ (for $1 \leq k \leq 6$) can be obtained from the database of interesting graphs at the \textit{House of Graphs}~\cite{CDG23} by searching for the keywords ``lonely edge''.


Following Corollary~\ref{cor:formula special edges}, larger elements of $\mathcal{U}_k$ can be constructed by starting with a smallest graph in $\mathcal{U}_k$ and repeatedly choosing a vertex $v$ with $p_v=0$ (if it exists) and replacing it by a triangle. Note that we will prove in the next section that this is the only way to construct larger graphs in $\mathcal{U}_4$.

In what follows we give a characterization of $\mathcal{U}_4$ and $\mathcal{U}_3$, and some results regarding $\mathcal{U}_2$ and $\mathcal{U}_1$.

\subsection{Characterization of \texorpdfstring{$\mathcal{U}_4$}{U4}}
Let $G$ be an arbitrary cubic graph with a triangle $T = x_1x_2x_3$ and let $s \in \{1,2,3\}$.
We will call the cubic graph obtained by replacing the vertex $x_s$ with a new triangle, a \emph{$T$-extension of index} $s$ of $G$.  
When we need to iterate this process, we will label the vertices of the new triangle again as $x_1,x_2,x_3$, and we will do it in such a way that, for both indices $j \neq s$, the new vertex $x_j$ is adjacent in the obtained graph to the vertex corresponding to $x_j$ in the original graph $G$.  Figure~\ref{fig:T-extension example} illustrates the notation we just introduced with an example.

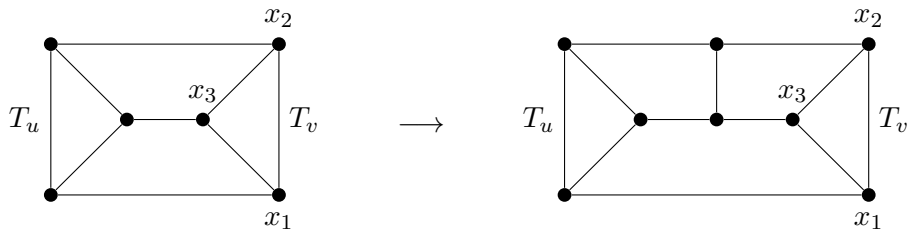
\begin{figure}[!htb]
    \centering
    \tikzstyle{vertex} = [fill, black, circle, scale=0.5]
        \def\Arrow{\raisebox{10\height}{$\longrightarrow$\,\,}}

    \begin{tikzpicture}
        \node[vertex] (1) at (0,1) {};
        \node[vertex] (2) at (0,-1) {};
        \node[vertex] (3) at (1,0) {};
        \node[vertex, label=$x_3$] (4) at (2,0) {};
        \node[vertex, label=$x_2$] (5) at (3,1) {};
        \node[vertex, label=below:$x_1$] (6) at (3,-1) {};
        \draw[] (1) to node[midway, left] {$T_u$} (2);
        \draw[] (1) to (3);
        \draw (1) to (5);
        \draw[] (2) to (3);
        \draw (2) to (6);
        \draw (3) to (4);
        \draw[] (4) to (5);
        \draw[] (4) to (6);
        \draw[] (5) to node[midway, right] {$T_v$} (6) ;
    \end{tikzpicture}
    \qquad\Arrow\qquad
    \begin{tikzpicture}
        \node[vertex] (1) at (0,1) {};
        \node[vertex] (2) at (0,-1) {};
        \node[vertex] (3) at (1,0) {};
        \node[vertex, label=$x_3$] (4) at (3,0) {};
        \node[vertex, label=$x_2$] (5) at (4,1) {};
        \node[vertex, label=below:$x_1$] (6) at (4,-1) {};
        \node[vertex] (7) at (2,1) {};
        \node[vertex] (8) at (2,0) {};
        \draw[] (1) to node[midway, left] {$T_u$} (2);
        \draw[] (1) to (3);
        \draw (1) to (7) to (5);
        \draw[] (2) to (3);
        \draw (2) to (6);
        \draw (3) to (8) to (4);
        \draw[] (4) to (5);
        \draw[] (4) to (6);
        \draw[] (5) to node[midway, right] {$T_v$} (6) ;
        \draw (7) to (8);
    \end{tikzpicture}
    
    \caption{Applying a $T_v$-extension of index $1$ to the prism for a given triangle $T_v$.}
    \label{fig:T-extension example}
\end{figure}




In particular, let $Pr$ be the prism graph with $6$ vertices obtained by replacing an arbitrary vertex of $K_4$ with a triangle. Note that $Pr$ has exactly two vertex-disjoint triangles, say $T_u$ and $T_v$ (see Figure~\ref{fig:T-extension example}).

A graph is an \emph{extended prism} if it is either the prism $Pr$ or if it is obtained by applying a $T$-extension to a smaller extended prism $P'$, where $T$ is a triangle of $P'$. Let $m\ge1$ and consider the ordered list $(s_1,...,s_m)$ such that $s_i \in \{1,2,3\}$ for all $i=1,\ldots,m$. 
An extended prism \emph{admits the extension pattern} $(s_1,\dots,s_m)$ if it can be obtained by starting from $Pr$
and applying consecutively a sequence of
$T_v$-extensions of indices $s_1,\ldots,s_m$, each time updating the triangle $T_v$ with the new triangle added as explained before. Note that an extended prism can admit multiple extension patterns. Throughout the paper we will also use the terminology "extended prism \emph{with} extension pattern $(s_1,\dots,s_m)$". 


By symmetry, observe that applying $T_u$-extensions to $Pr$ with reversed extension pattern $(s_m,\ldots,s_1)$, and each time updating $T_u$ with the new triangle, gives rise to the same graph.



\begin{observation}\label{obs:subsequence}
By previous considerations, it follows that if the extension pattern of an extended prism $Pr'$ (or any other permutation of its three labels) appears as a consecutive subsequence of the extension pattern, or of its reverse, of an extended prism $Pr''$, then $Pr''$ is a descendant of $Pr'$.
\end{observation}


We can now characterize the set $\mathcal{U}_4$ as follows.

\begin{theorem}\label{thm:Characterization_U4}
    A graph $G$ belongs to $\mathcal{U}_4$ if and only if $G$ is an extended prism which admits extension pattern $(1,2)$ or $(\underbrace{1,\dots,1}_{k \text{ times}})$, for $k\ge2$.
\end{theorem}


\begin{proof}
    The only Klee-graphs on at most $8$ vertices are those depicted in Figures~\ref{fig:5lonelyEdges} and~\ref{fig:6lonelyEdges}. In particular there is only one graph with $5$ lonely edges and it has $8$ vertices. From this graph we can obtain three non-isomorphic children. One of them is the tricorn graph, see Figure~\ref{fig:3lonelyEdges}, which has $3$ lonely edges. The other two are extended prisms $Pr_1$ and $Pr_2$ with extension patterns $(1,2)$ and $(1,1)$ respectively, see Figure~\ref{fig:smallest_U_4}. (The same two graphs are also depicted in Figure~\ref{fig:4lonelyEdges}: here the reader can better see how they are constructed from the unique graph in $\mathcal{U}_5$.) Both of them are in $\mathcal{U}_4$.

    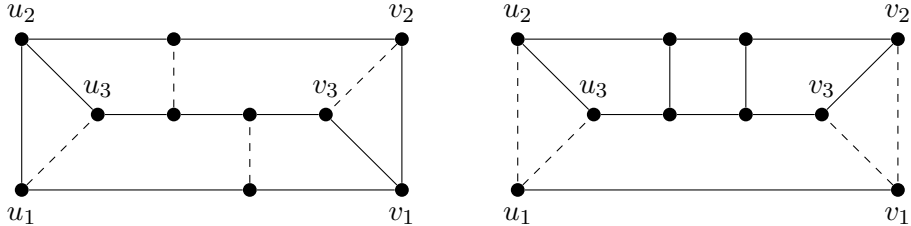
\begin{figure}[!htb]
    \centering
    \begin{tikzpicture}[scale=1, yscale=-1]
	\node[circle, fill, scale=0.5, label=$u_2$] (0) at (0, 0) {}; 
	\node[circle, fill, scale=0.5, label=below:$u_1$] (1) at (0, 2) {}; 
	\node[circle, fill, scale=0.5, label=$u_3$] (2) at (1, 1) {}; 
	\node[circle, fill, scale=0.5, label=$v_2$] (3) at (5, 0) {}; 
	\node[circle, fill, scale=0.5, label=below:$v_1$] (4) at (5, 2) {}; 
	\node[circle, fill, scale=0.5, label=$v_3$] (5) at (4, 1) {}; 
	\node[circle, fill, scale=0.5] (6) at (2.0, 1) {}; 
	\node[circle, fill, scale=0.5] (7) at (2.0, 0) {}; 
	\node[circle, fill, scale=0.5] (8) at (3.0, 2) {}; 
	\node[circle, fill, scale=0.5] (9) at (3.0, 1) {}; 

	\draw[black] (0) to (1);
	\draw[black] (0) to (2);
	\draw[black] (0) to (7);
	\draw[black,dashed] (1) to (2);
	\draw[black] (1) to (8);
	\draw[black] (2) to (6);
	\draw[black] (3) to (4);
	\draw[black,dashed] (3) to (5);
	\draw[black] (3) to (7);
	\draw[black] (4) to (5);
	\draw[black] (4) to (8);
	\draw[black] (5) to (9);
	\draw[black,dashed] (6) to (7);
	\draw[black] (6) to (9);
	\draw[black,dashed] (8) to (9);
\end{tikzpicture}\qquad
    \begin{tikzpicture}[scale=1, yscale=-1]
	\node[circle, fill, scale=0.5, label=$u_2$] (0) at (0, 0) {}; 
	\node[circle, fill, scale=0.5, label=below:$u_1$] (1) at (0, 2) {}; 
	\node[circle, fill, scale=0.5, label=$u_3$] (2) at (1, 1) {}; 
	\node[circle, fill, scale=0.5, label=$v_2$] (3) at (5, 0) {}; 
	\node[circle, fill, scale=0.5, label=below:$v_1$] (4) at (5, 2) {}; 
	\node[circle, fill, scale=0.5, label=$v_3$] (5) at (4, 1) {}; 
	\node[circle, fill, scale=0.5] (6) at (2.0, 1) {}; 
	\node[circle, fill, scale=0.5] (7) at (2.0, 0) {}; 
	\node[circle, fill, scale=0.5] (8) at (3.0, 1) {}; 
	\node[circle, fill, scale=0.5] (9) at (3.0, 0) {}; 

	\draw[black,dashed] (0) to (1);
	\draw[black] (0) to (2);
	\draw[black] (0) to (7);
	\draw[black,dashed] (1) to (2);
	\draw[black] (1) to (4);
	\draw[black] (2) to (6);
	\draw[black,dashed] (3) to (4);
	\draw[black] (3) to (5);
	\draw[black] (3) to (9);
	\draw[black,dashed] (4) to (5);
	\draw[black] (5) to (8);
	\draw[black] (6) to (7);
	\draw[black] (6) to (8);
	\draw[black] (7) to (9);
	\draw[black] (8) to (9);
\end{tikzpicture}
    \caption{The smallest graphs in $\mathcal{U}_4$. On the left we have the extended prism with extension pattern $(1,2)$, denoted by $Pr_1$ and on the right the extended prism with extension pattern $(1,1)$ denoted by $Pr_2$. }
    \label{fig:smallest_U_4}
\end{figure}
    
    We now prove that the extended prisms admitting extension pattern $(\underbrace{1,\dots,1}_{k \text{ times}})$, for all $k\ge3$, are the only other graphs in $\mathcal{U}_4$. We use induction on the number of vertices $n = 2k+6$ (where $k$ denotes the size of the extension pattern).



    
    For $k=3$,
    it follows by previous arguments that a graph on $12$ vertices in $\mathcal{U}_4$ must be a child of a graph in $\mathcal{U}_4$ on $10$ vertices, namely of either $Pr_1$ or $Pr_2$.
    By using Lemma~\ref{lem:blowUpToTriangle} it is easy to check that $Pr_1$ has no children in $\mathcal{U}_4$ (indeed, for every vertex $v$ of $Pr_1$ there is a join rooted in $v$ containing a lonely edge). Moreover, $Pr_2$ has  exactly one child in $\mathcal{U}_4$. Indeed, note that the lonely edges of $Pr_2$ are $u_1u_2$, $u_1u_3$, $v_1v_2$ and $v_1v_3$ (i.e.\ the edges of $T_u$ and $T_v$ incident with $u_1$ and $v_1$, respectively). Therefore, again by Lemma~\ref{lem:blowUpToTriangle}, its only child in $\mathcal{U}_4$ is the extended prism with extension pattern $(1,1,1)$. This graph still has as lonely edges $u_1u_2$, $u_1u_3$, $v_1v_2$ and $v_1v_3$.  
    
    We now show by induction the same holds for $k>3$. 
    By previous arguments a graph $H$ in $\mathcal{U}_4$ on $n\geq 12$ vertices must be a child of a graph in $\mathcal{U}_4$ on $n-2$ vertices. 
    By the inductive hypothesis, there is exactly one such graph. Namely, the extended prism $P'$ with extension pattern $(\underbrace{1,\dots,1}_{k-1 \text{ times}})$. Now, it is easy to check that the only vertex $v$ of $P'$ such that $(P')^v$ is in $\mathcal{U}_4$ is one of the two vertices adjacent with two lonely edges. Replacing this vertex with a triangle we obtain an extended prism with the desired pattern. Moreover, again the lonely edges of this new graph are $u_1u_2$, $u_1u_3$, $v_1v_2$ and $v_1v_3$.
\end{proof}

\subsection{Characterization of \texorpdfstring{$\mathcal{U}_3$}{U3}}

By using Lemma~\ref{lem:blowUpToTriangle}, a direct check shows that the tricorn graph belongs to $\mathcal{U}_3$ and has exactly three descendants in $\mathcal{U}_3$, see Figure~\ref{fig:tricorn_descendants}.

\begin{figure}[!htb]
    \centering
    \tikzstyle{vertex} = [fill, black, circle, scale=0.5]
\begin{tikzpicture}[rotate=0]
    \node[vertex] (0) at (0,0) {};
    \node[vertex] (c1) at (0,-1) {};
    \node[vertex] (c2) at (0.86602540378,-1.5) {};
    \node[vertex] (c3) at (-0.86602540378,-1.5) {};

    \coordinate (c1') at (0,-0.5);
    \coordinate (c2') at (0.86602540378,-1.5);
    \coordinate (c3') at (-0.86602540378,-1.5);
    \coordinate (c4') at (0.56602540378,-1);
    \coordinate (c5') at (-0.56602540378,-1);
    
    \draw (0) to (c1);
    \draw (c1) to (c2);
    \draw (c1) to (c3);
    \draw[dashed] (c2) to (c3);

    \path let \p1 = (c1), \p2 = (c2), \p3 = (c3) in
        {[rotate=240]   coordinate (l1) at (\p1)
                        coordinate (l2) at (\p2)
                        coordinate (l3) at (\p3)};

    \node[vertex] at (l1) {};
    \node[vertex] at (l2) {};
    \node[vertex] at (l3) {};

    \draw (0) to (l1);
    \draw (l1) to (l2);
    \draw (l1) to (l3);
    \draw[dashed] (l2) to (l3);

    \path let \p1 = (c1'), \p2 = (c2'), \p3 = (c3'), \p4 = (c4'), \p5 = (c5') in
        {[rotate=120]   coordinate (r1) at (\p1)
                        coordinate (r2) at (\p2)
                        coordinate (r3) at (\p3)
                        coordinate (r4) at (\p4)
                        coordinate (r5) at (\p5)};

    \node[vertex] at (r1) {};
    \node[vertex] at (r2) {};
    \node[vertex] at (r3) {};
    \node[vertex] at (r4) {};
    \node[vertex] at (r5) {};
    
    \draw (0) to (r1);
    \draw (r1) to (r4);
    \draw (r1) to (r5);
    \draw[dashed] (r2) to (r3);
    \draw (r2) to (r4);
    \draw (r3) to (r5);
    \draw (r4) to (r5);

    \draw (r2) to (l3);
    \draw (l2) to (c3);
    \draw (c2) to (r3);
\end{tikzpicture}\qquad
\begin{tikzpicture}
    \node[vertex] (0) at (0,0) {};
    \node[vertex] (c1) at (0,-1) {};
    \node[vertex] (c2) at (0.86602540378,-1.5) {};
    \node[vertex] (c3) at (-0.86602540378,-1.5) {};

    \coordinate (c1') at (0,-0.5);
    \coordinate (c2') at (0.86602540378,-1.5);
    \coordinate (c3') at (-0.86602540378,-1.5);
    \coordinate (c4') at (0.56602540378,-1);
    \coordinate (c5') at (-0.56602540378,-1);
    
    \draw (0) to (c1);
    \draw (c1) to (c2);
    \draw (c1) to (c3);
    \draw[dashed] (c2) to (c3);

    \path let \p1 = (c1'), \p2 = (c2'), \p3 = (c3'), \p4 = (c4'), \p5 = (c5') in
        {[rotate=240]   coordinate (l1) at (\p1)
                        coordinate (l2) at (\p2)
                        coordinate (l3) at (\p3)
                        coordinate (l4) at (\p4)
                        coordinate (l5) at (\p5)};

    \node[vertex] at (l1) {};
    \node[vertex] at (l2) {};
    \node[vertex] at (l3) {};
    \node[vertex] at (l4) {};
    \node[vertex] at (l5) {};
    
    \draw (0) to (l1);
    \draw (l1) to (l4);
    \draw (l1) to (l5);
    \draw[dashed] (l2) to (l3);
    \draw (l2) to (l4);
    \draw (l3) to (l5);
    \draw (l4) to (l5);

    \draw (r2) to (l3);
    \draw (l2) to (c3);
    \draw (c2) to (r3);

    \path let \p1 = (c1'), \p2 = (c2'), \p3 = (c3'), \p4 = (c4'), \p5 = (c5') in
        {[rotate=120]   coordinate (r1) at (\p1)
                        coordinate (r2) at (\p2)
                        coordinate (r3) at (\p3)
                        coordinate (r4) at (\p4)
                        coordinate (r5) at (\p5)};

    \node[vertex] at (r1) {};
    \node[vertex] at (r2) {};
    \node[vertex] at (r3) {};
    \node[vertex] at (r4) {};
    \node[vertex] at (r5) {};
    
    \draw (0) to (r1);
    \draw (r1) to (r4);
    \draw (r1) to (r5);
    \draw[dashed] (r2) to (r3);
    \draw (r2) to (r4);
    \draw (r3) to (r5);
    \draw (r4) to (r5);

    \draw (r2) to (l3);
    \draw (l2) to (c3);
    \draw (c2) to (r3);
\end{tikzpicture}\qquad
\begin{tikzpicture}
    \node[vertex] (0) at (0,0) {};

    \coordinate (c1') at (0,-0.5);
    \coordinate (c2') at (0.86602540378,-1.5);
    \coordinate (c3') at (-0.86602540378,-1.5);
    \coordinate (c4') at (0.56602540378,-1);
    \coordinate (c5') at (-0.56602540378,-1);
    
    \path let \p1 = (c1'), \p2 = (c2'), \p3 = (c3'), \p4 = (c4'), \p5 = (c5') in
        {[rotate=0]   coordinate (c1) at (\p1)
                        coordinate (c2) at (\p2)
                        coordinate (c3) at (\p3)
                        coordinate (c4) at (\p4)
                        coordinate (c5) at (\p5)};

    \node[vertex] at (c1) {};
    \node[vertex] at (c2) {};
    \node[vertex] at (c3) {};
    \node[vertex] at (c4) {};
    \node[vertex] at (c5) {};
    
    \draw (0) to (c1);
    \draw (c1) to (c4);
    \draw (c1) to (c5);
    \draw[dashed] (c2) to (c3);
    \draw (c2) to (c4);
    \draw (c3) to (c5);
    \draw (c4) to (c5);

    \path let \p1 = (c1'), \p2 = (c2'), \p3 = (c3'), \p4 = (c4'), \p5 = (c5') in
        {[rotate=240]   coordinate (l1) at (\p1)
                        coordinate (l2) at (\p2)
                        coordinate (l3) at (\p3)
                        coordinate (l4) at (\p4)
                        coordinate (l5) at (\p5)};

    \node[vertex] at (l1) {};
    \node[vertex] at (l2) {};
    \node[vertex] at (l3) {};
    \node[vertex] at (l4) {};
    \node[vertex] at (l5) {};
    
    \draw (0) to (l1);
    \draw (l1) to (l4);
    \draw (l1) to (l5);
    \draw[dashed] (l2) to (l3);
    \draw (l2) to (l4);
    \draw (l3) to (l5);
    \draw (l4) to (l5);

    \draw (r2) to (l3);
    \draw (l2) to (c3);
    \draw (c2) to (r3);

    \path let \p1 = (c1'), \p2 = (c2'), \p3 = (c3'), \p4 = (c4'), \p5 = (c5') in
        {[rotate=120]   coordinate (r1) at (\p1)
                        coordinate (r2) at (\p2)
                        coordinate (r3) at (\p3)
                        coordinate (r4) at (\p4)
                        coordinate (r5) at (\p5)};

    \node[vertex] at (r1) {};
    \node[vertex] at (r2) {};
    \node[vertex] at (r3) {};
    \node[vertex] at (r4) {};
    \node[vertex] at (r5) {};
    
    \draw (0) to (r1);
    \draw (r1) to (r4);
    \draw (r1) to (r5);
    \draw[dashed] (r2) to (r3);
    \draw (r2) to (r4);
    \draw (r3) to (r5);
    \draw (r4) to (r5);

    \draw (r2) to (l3);
    \draw (l2) to (c3);
    \draw (c2) to (r3);
\end{tikzpicture}
    \caption{All descendants in $\mathcal{U}_3$ of the tricorn. The lonely edges are dashed.}
    \label{fig:tricorn_descendants}
\end{figure}
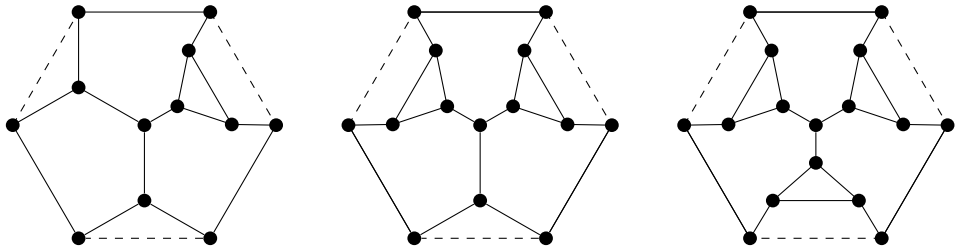

     


\begin{lemma}
\label{lem:3 triangles implie descendant of tricorn}
    If $G$ is a Klee-graph of order at least 10 with three triangles, then $G$ is the tricorn or a descendant of the tricorn.
\end{lemma}
\begin{proof}
    We argue by induction on $|V(G)|$. For $|V(G)|=10$ the statement is easy to check. Thus, assume that $G$ has at least 12 vertices. Now, we show that we can always find a triangle whose contraction gives a smaller Klee-graph with three triangles, and then we can apply the inductive assumption. 
    Note that every two triangles of a Klee-graph that is not $K_4$ are vertex-disjoint, since every Klee-graph is 3-edge-connected. Let $T_1,T_2,T_3$ be three triangles of $G$. For every $i \in \{1,2,3\}$ contract $T_i$ to a new vertex $x_i$ and let $G'$ be the resulting graph of order $|V(G)|-6$. By Lemma~\ref{lem:G^v KG then G KG}, $G'$ is a Klee-graph, which has two disjoint triangles $T_1', T_2'$ by Observation~\ref{obs:trNotContainingVertex}. Clearly, one of these triangles contains at most one vertex of $\{x_1,x_2,x_3\}$. Without loss of generality we assume that $T_1'$ does not contain $x_2,x_3$. Let $G''$ be the graph obtained from $G'$ by replacing $x_2$ and $x_3$ by a triangle. Note that $G''$ is isomorphic to the graph obtained from $G$ by contracting $T_1$ to a new vertex. The graph $G''$ is a Klee-graph by Lemma~\ref{lem:G^v KG then G KG} and has three triangles, namely $T_1'$ and the two new triangles obtained by replacing $x_2,x_3$. Therefore, $G''$ is the tricorn or a descendant of the tricorn by inductive hypothesis and hence, $G$ is a descendant of the tricorn.
\end{proof}

Let ${\cal W}_3$ be the subset of ${\cal U}_3$ consisting of the following graphs: the tricorn graph and its descendants in $\mathcal{U}_3$, and the extended prism with extension pattern $(1,2,1)$. See Figure~\ref{fig:313}. 

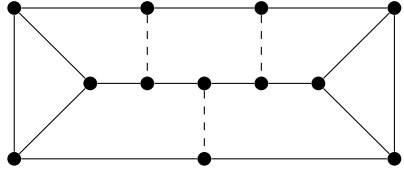
\begin{figure}[!htb]
    \centering
    \begin{tikzpicture}[scale=1]
	\node[circle, fill, scale=0.5] (0) at (0, 0) {}; 
	\node[circle, fill, scale=0.5] (1) at (0, 2) {}; 
	\node[circle, fill, scale=0.5] (2) at (1, 1) {}; 
	\node[circle, fill, scale=0.5] (3) at (5, 0) {}; 
	\node[circle, fill, scale=0.5] (4) at (5, 2) {}; 
	\node[circle, fill, scale=0.5] (5) at (4, 1) {}; 
	\node[circle, fill, scale=0.5] (6) at (1.75, 2) {}; 
	\node[circle, fill, scale=0.5] (7) at (1.75, 1) {}; 
	\node[circle, fill, scale=0.5] (8) at (2.5, 1) {}; 
	\node[circle, fill, scale=0.5] (9) at (2.5, 0) {}; 
	\node[circle, fill, scale=0.5] (10) at (3.25, 2) {}; 
	\node[circle, fill, scale=0.5] (11) at (3.25, 1) {}; 

	\draw[black] (0) to (1);
	\draw[black] (0) to (2);
	\draw[black] (0) to (9);
	\draw[black] (1) to (2);
	\draw[black] (1) to (6);
	\draw[black] (2) to (7);
	\draw[black] (3) to (4);
	\draw[black] (3) to (5);
	\draw[black] (3) to (9);
	\draw[black] (4) to (5);
	\draw[black] (4) to (10);
	\draw[black] (5) to (11);
	\draw[black, dashed] (6) to (7);
	\draw[black] (6) to (10);
	\draw[black] (7) to (8);
	\draw[black, dashed] (8) to (9);
	\draw[black] (8) to (11);
	\draw[black, dashed] (10) to (11);
\end{tikzpicture}
    \caption{An extended prism with extension pattern $(1,2,1)$. Its lonely edges are dashed.}
    \label{fig:313}
\end{figure}

\begin{theorem}\label{thm:Characterization_U3}
A graph $G$ belongs to $\mathcal{U}_3$ if and only if  either $G\in {\cal W}_3$ or $G$ is an extended prism which admits extension pattern $( \underbrace{1,\dots,1}_{h-1 \text{ times}}, 3, \underbrace{2,\dots,2}_{k-h \text{ times}})$, where $k\geq 3$ and $\frac{k+1}{2}\leq h\leq k$.

\end{theorem}


\begin{proof}
First, we prove that if $G \in \mathcal{U}_3$ then it is either in $\mathcal{W}_3$ or it is an extended prism with extension pattern as described in the statement.

Let $G \in \mathcal{U}_3 \setminus {\cal W}_3$. Recall that the tricorn only has three descendants in $\mathcal{U}_3$, all of them belong to ${\cal W}_3$. Hence by Lemma~\ref{lem:3 triangles implie descendant of tricorn}, $G$ has at most two triangles. As a consequence, $G$ is an extended prism, and we need to prove that the only possible extension patterns are the ones in the statement.

A direct check shows that the extended prisms with extension patterns $(1,1,3)$ and $(1,3,2)$ are the smallest graphs in $\mathcal{U}_3 \setminus \mathcal{W}_3$, which we call $Pr_3$ and $Pr_4$, respectively. So, we can assume that $G$ has more than $12$ vertices.


    \newcounter{Number_of_patterns_in_first_figure}

    \begin{figure}[!htb]
    \def\Arrow{\raisebox{6\height}{$\longrightarrow$\,\,}}
    \setlength{\abovecaptionskip}{-3pt} 
        \begin{enumerate}
        \setlength{\itemsep}{-18pt}

        

        \item \parbox{\linewidth}{
            \begin{tikzpicture}[scale=1]
	\node[circle, fill, scale=0.5] (0) at (0, 0) {}; 
	\node[circle, fill, scale=0.5] (1) at (0, 2) {}; 
	\node[circle, fill, scale=0.5] (2) at (1, 1) {}; 
	\node[circle, fill, scale=0.5] (3) at (5, 0) {}; 
	\node[circle, fill, scale=0.5] (4) at (5, 2) {}; 
	\node[circle, fill, scale=0.5] (5) at (4, 1) {}; 
	\node[circle, fill, scale=0.5] (6) at (1.75, 1) {}; 
	\node[circle, fill, scale=0.5] (7) at (1.75, 0) {}; 
	\node[circle, fill, scale=0.5] (8) at (2.5, 1) {}; 
	\node[circle, fill, scale=0.5] (9) at (2.5, 0) {}; 
	\node[circle, fill, scale=0.5] (10) at (3.25, 2) {}; 
	\node[circle, fill, scale=0.5] (11) at (3.25, 1) {}; 

    \draw[white] (1) circle[radius=0.2];
    \draw[white] (3) circle[radius=0.2];

	\draw[black, dashed] (0) to (1);
	\draw[black] (0) to (2);
	\draw[black, line width=\w] (0.center) to (7.center);
	\draw[black, line width=\w] (1.center) to (2.center);
	\draw[black] (1) to (10);
	\draw[black] (2) to (6);
	\draw[black, line width=\w] (3) to (4);
	\draw[black, line width=\w, dashed] (3.center) to (5.center);
	\draw[black, line width=\w] (3) to (9);
	\draw[black] (4) to (5);
	\draw[black] (4) to (10);
	\draw[black] (5) to (11);
	\draw[black] (6) to (7);
	\draw[black, line width=\w] (6.center) to (8.center);
	\draw[black] (7) to (9);
	\draw[black] (8) to (9);
	\draw[black] (8) to (11);
	\draw[black, line width=\w, dashed] (10.center) to (11.center);

     \draw[thick] (3) circle[radius=0.2];

\end{tikzpicture}
            \Arrow
            \begin{tikzpicture}[scale=1]
	\node[circle, fill, scale=0.5] (0) at (0, 0) {}; 
	\node[circle, fill, scale=0.5] (1) at (0, 2) {}; 
	\node[circle, fill, scale=0.5] (2) at (1, 1) {}; 
	\node[circle, fill, scale=0.5] (3) at (5, 0) {}; 
	\node[circle, fill, scale=0.5] (4) at (5, 2) {}; 
	\node[circle, fill, scale=0.5] (5) at (4, 1) {}; 
	\node[circle, fill, scale=0.5] (6) at (1.6, 1) {}; 
	\node[circle, fill, scale=0.5] (7) at (1.6, 0) {}; 
	\node[circle, fill, scale=0.5] (8) at (2.2, 1) {}; 
	\node[circle, fill, scale=0.5] (9) at (2.2, 0) {}; 
	\node[circle, fill, scale=0.5] (10) at (2.8, 2) {}; 
	\node[circle, fill, scale=0.5] (11) at (2.8, 1) {}; 
	\node[circle, fill, scale=0.5] (12) at (3.4, 2) {}; 
	\node[circle, fill, scale=0.5] (13) at (3.4, 1) {}; 

    \draw[white] (1) circle[radius=0.2];
    \draw[white] (3) circle[radius=0.2];

	\draw[black, dashed] (0) to (1);
	\draw[black] (0) to (2);
	\draw[black] (0) to (7);
	\draw[black] (1) to (2);
	\draw[black] (1) to (10);
	\draw[black] (2) to (6);
	\draw[black, dashed] (3) to (4);
	\draw[black] (3) to (5);
	\draw[black] (3) to (9);
	\draw[black] (4) to (5);
	\draw[black] (4) to (12);
	\draw[black] (5) to (13);
	\draw[black] (6) to (7);
	\draw[black] (6) to (8);
	\draw[black] (7) to (9);
	\draw[black] (8) to (9);
	\draw[black] (8) to (11);
	\draw[black] (10) to (11);
	\draw[black] (10) to (12);
	\draw[black] (11) to (13);
	\draw[black] (12) to (13);
\end{tikzpicture}
        }
        \label{pat:1133}



        \item \parbox{\linewidth}{
            \begin{tikzpicture}[scale=1]
	\node[circle, fill, scale=0.5] (0) at (0, 0) {}; 
	\node[circle, fill, scale=0.5] (1) at (0, 2) {}; 
	\node[circle, fill, scale=0.5] (2) at (1, 1) {}; 
	\node[circle, fill, scale=0.5] (3) at (5, 0) {}; 
	\node[circle, fill, scale=0.5] (4) at (5, 2) {}; 
	\node[circle, fill, scale=0.5] (5) at (4, 1) {}; 
	\node[circle, fill, scale=0.5] (6) at (1.75, 1) {}; 
	\node[circle, fill, scale=0.5] (7) at (1.75, 0) {}; 
	\node[circle, fill, scale=0.5] (8) at (2.5, 2) {}; 
	\node[circle, fill, scale=0.5] (9) at (2.5, 1) {}; 
	\node[circle, fill, scale=0.5] (10) at (3.25, 2) {}; 
	\node[circle, fill, scale=0.5] (11) at (3.25, 0) {}; 

    \draw[white] (1) circle[radius=0.2];
    \draw[white] (3) circle[radius=0.2];

	\draw[black] (0) to (1);
	\draw[black, line width=\w] (0.center) to (2.center);
	\draw[black] (0) to (7);
	\draw[black, line width=\w, dashed] (1.center) to (2.center);
	\draw[black] (1) to (8);
	\draw[black, line width=\w] (2.center) to (6.center);
	\draw[black] (3) to (4);
	\draw[black, line width=\w] (3.center) to (5.center);
	\draw[black] (3) to (11);
	\draw[black, dashed] (4) to (5);
	\draw[black, line width=\w] (4.center) to (10.center);
	\draw[black] (5) to (9);
	\draw[black] (6) to (7);
	\draw[black] (6) to (9);
	\draw[black, line width=\w] (7.center) to (11.center);
	\draw[black, line width=\w, dashed] (8.center) to (9.center);
	\draw[black] (8) to (10);
	\draw[black] (10) to (11);

     \draw[thick] (2) circle[radius=0.2];
     

\end{tikzpicture}
            \Arrow
            \begin{tikzpicture}[scale=1]
	\node[circle, fill, scale=0.5] (0) at (0, 0) {}; 
	\node[circle, fill, scale=0.5] (1) at (0, 2) {}; 
	\node[circle, fill, scale=0.5] (2) at (1, 1) {}; 
	\node[circle, fill, scale=0.5] (3) at (5, 0) {}; 
	\node[circle, fill, scale=0.5] (4) at (5, 2) {}; 
	\node[circle, fill, scale=0.5] (5) at (4, 1) {}; 
	\node[circle, fill, scale=0.5] (6) at (1.6, 2) {}; 
	\node[circle, fill, scale=0.5] (7) at (1.6, 0) {}; 
	\node[circle, fill, scale=0.5] (8) at (2.2, 1) {}; 
	\node[circle, fill, scale=0.5] (9) at (2.2, 0) {}; 
	\node[circle, fill, scale=0.5] (10) at (2.8, 2) {}; 
	\node[circle, fill, scale=0.5] (11) at (2.8, 1) {}; 
	\node[circle, fill, scale=0.5] (12) at (3.4, 2) {}; 
	\node[circle, fill, scale=0.5] (13) at (3.4, 0) {}; 

    \draw[white] (1) circle[radius=0.2];
    \draw[white] (3) circle[radius=0.2];

	\draw[black] (0) to (1);
	\draw[black, dashed] (0) to (2);
	\draw[black] (0) to (7);
	\draw[black] (1) to (2);
	\draw[black] (1) to (6);
	\draw[black] (2) to (8);
	\draw[black] (3) to (4);
	\draw[black] (3) to (5);
	\draw[black] (3) to (13);
	\draw[black, dashed] (4) to (5);
	\draw[black] (4) to (12);
	\draw[black] (5) to (11);
	\draw[black] (6) to (7);
	\draw[black] (6) to (10);
	\draw[black] (7) to (9);
	\draw[black] (8) to (9);
	\draw[black] (8) to (11);
	\draw[black] (9) to (13);
	\draw[black] (10) to (11);
	\draw[black] (10) to (12);
	\draw[black] (12) to (13);
\end{tikzpicture}
        }
        \label{pat:2132}
        
        \item \parbox{\linewidth}{
            \begin{tikzpicture}[scale=1]
	\node[circle, fill, scale=0.5] (0) at (0, 0) {}; 
	\node[circle, fill, scale=0.5] (1) at (0, 2) {}; 
	\node[circle, fill, scale=0.5] (2) at (1, 1) {}; 
	\node[circle, fill, scale=0.5] (3) at (5, 0) {}; 
	\node[circle, fill, scale=0.5] (4) at (5, 2) {}; 
	\node[circle, fill, scale=0.5] (5) at (4, 1) {}; 
	\node[circle, fill, scale=0.5] (6) at (1.75, 1) {}; 
	\node[circle, fill, scale=0.5] (7) at (1.75, 0) {}; 
	\node[circle, fill, scale=0.5] (8) at (2.5, 2) {}; 
	\node[circle, fill, scale=0.5] (9) at (2.5, 1) {}; 
	\node[circle, fill, scale=0.5] (10) at (3.25, 1) {}; 
	\node[circle, fill, scale=0.5] (11) at (3.25, 0) {};

    \draw[white] (1) circle[radius=0.2];
    \draw[white] (3) circle[radius=0.2];

	\draw[black] (0) to (1);
	\draw[black] (0) to (2);
	\draw[black] (0) to (7);
	\draw[black] (1) to (2);
	\draw[black] (1) to (8);
	\draw[black] (2) to (6);
	\draw[black] (3) to (4);
	\draw[black] (3) to (5);
	\draw[black] (3) to (11);
	\draw[black] (4) to (5);
	\draw[black] (4) to (8);
	\draw[black] (5) to (10);
	\draw[black, dashed] (6) to (7);
	\draw[black] (6) to (9);
	\draw[black] (7) to (11);
	\draw[black, dashed] (8) to (9);
	\draw[black] (9) to (10);
	\draw[black, dashed] (10) to (11);

    \draw[thick] (0) circle[radius=0.2];
    \draw[thick] (1) circle[radius=0.2];
    \draw[thick] (2) circle[radius=0.2];
    \draw[thick] (3) circle[radius=0.2];
    \draw[thick] (4) circle[radius=0.2];
    \draw[thick] (5) circle[radius=0.2];


\end{tikzpicture}
            \qquad
            \raisebox{2\height}{\parbox{0.4\linewidth}{Expanding any encircled vertex yields a graph in $\mathcal{U}_2$.}}
        }
        \label{pat:..131..}
        
        \setcounter{Number_of_patterns_in_first_figure}{\value{enumi}}
        \end{enumerate}
        \caption{On the right-hand side of the picture we show extended prisms with at most two lonely edges (dashed). They can be obtained by replacing the encircled vertex $v$ of the graph on the left-hand side with a triangle, where the thick edges represent joins rooted in $v$. The third point produces multiple bad patterns, i.e.\ for all $t\in\{1,2,3\}$ the extended prism with pattern $(2,1,2,t)$ has at most two lonely edges.}
        \label{fig:bad_patterns}
    \end{figure}
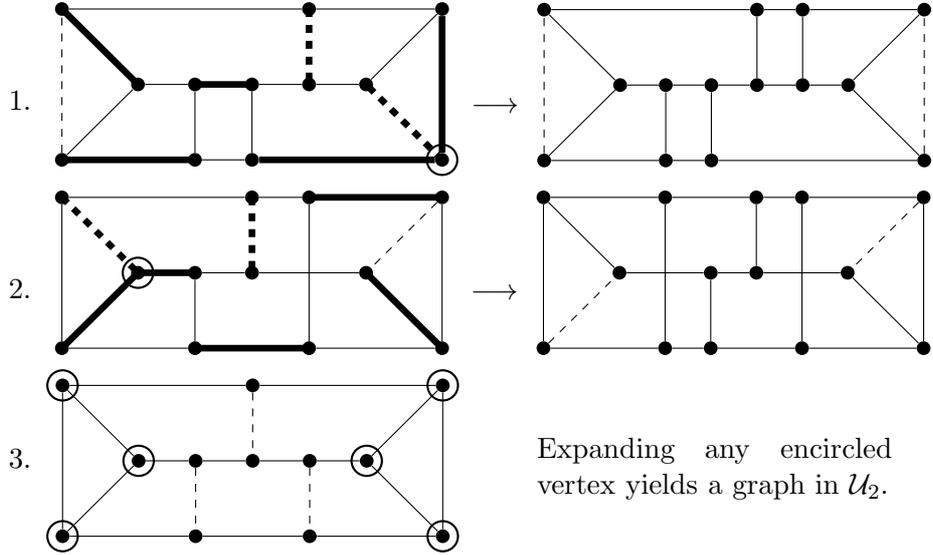
    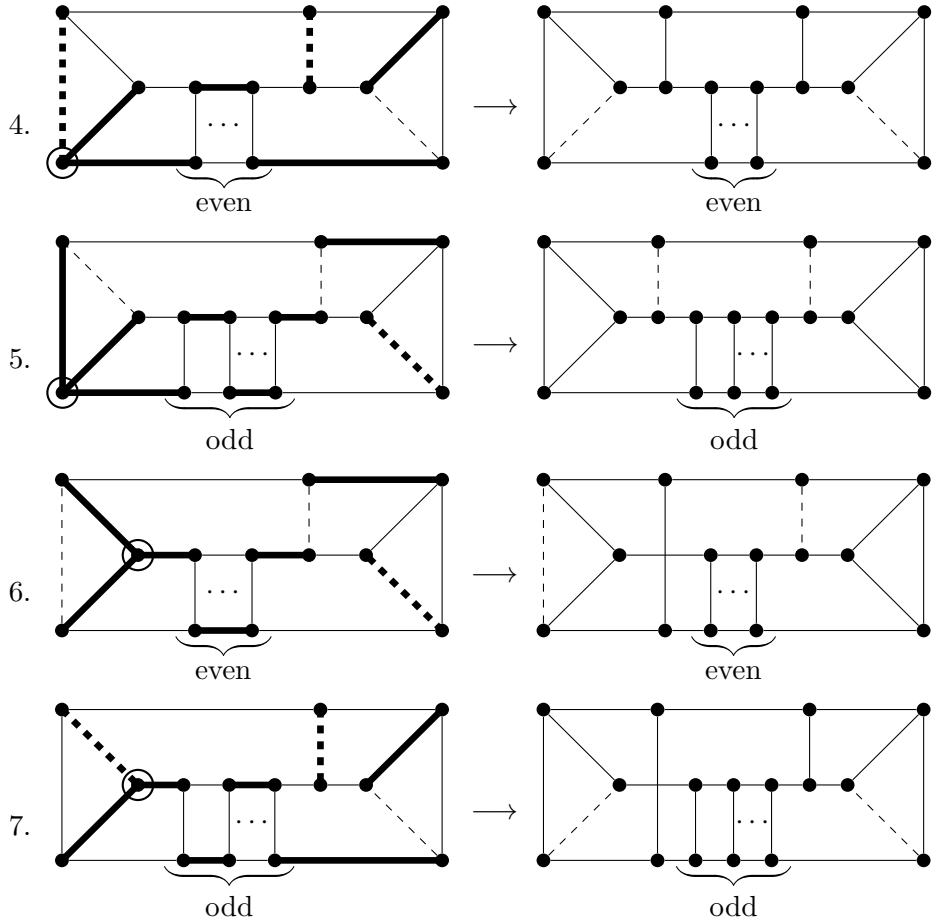
\begin{figure}[!htb]
        \def\Arrow{\raisebox{10\height}{$\longrightarrow$\,\,}}
        \setlength{\abovecaptionskip}{-18pt} 
        \begin{enumerate}
            \setcounter{enumi}{\value{Number_of_patterns_in_first_figure}}
            \setlength{\itemsep}{-18pt}
    
            \item \parbox{\linewidth}{
                \begin{tikzpicture}[scale=1, brace/.style={thick,decorate, decoration={calligraphic brace, amplitude=7pt,raise=0.5ex}}]
	\node[circle, fill, scale=0.5] (0) at (0, 0) {}; 
	\node[circle, fill, scale=0.5] (1) at (0, 2) {}; 
	\node[circle, fill, scale=0.5] (2) at (1, 1) {}; 
	\node[circle, fill, scale=0.5] (3) at (5, 0) {}; 
	\node[circle, fill, scale=0.5] (4) at (5, 2) {}; 
	\node[circle, fill, scale=0.5] (5) at (4, 1) {}; 
	\node[circle, fill, scale=0.5] (6) at (1.75, 1) {}; 
	\node[circle, fill, scale=0.5] (7) at (1.75, 0) {}; 
	\node[circle, fill, scale=0.5] (8) at (2.5, 1) {}; 
	\node[circle, fill, scale=0.5] (9) at (2.5, 0) {}; 
	\node[circle, fill, scale=0.5] (10) at (3.25, 2) {}; 
	\node[circle, fill, scale=0.5] (11) at (3.25, 1) {}; 
    \node at (2.15,0.5) {$\cdots$};

    \draw[white] (1) circle[radius=0.2];
    \draw[white] (3) circle[radius=0.2];

	\draw[black, line width=\w, dashed] (0.center) to (1.center);
	\draw[black, line width=\w] (0.center) to (2.center);
	\draw[black, line width=\w] (0.center) to (7.center);
	\draw[black] (1) to (2);
	\draw[black] (1) to (10);
	\draw[black] (2) to (6);
	\draw[black] (3) to (4);
	\draw[black, dashed] (3) to (5);
	\draw[black, line width=\w] (3.center) to (9.center);
	\draw[black, line width=\w] (4.center) to (5.center);
	\draw[black] (4) to (10);
	\draw[black] (5) to (11);
	\draw[black] (6) to (7);
	\draw[black, line width=\w] (6.center) to (8.center);
	\draw[black] (7) to (9);
	\draw[black] (8) to (9);
	\draw[black] (8) to (11);
	\draw[black, line width=\w, dashed] (10.center) to (11.center);

     \draw[thick] (0) circle[radius=0.2];
     \draw[brace] ($ (9) + (.25,0) $)-- node[below=2ex]{even} ($ (7) + (-.25,0) $);


\end{tikzpicture}
                \Arrow
                \begin{tikzpicture}[scale=1, brace/.style={thick,decorate, decoration={calligraphic brace, amplitude=7pt,raise=0.5ex}}]
	\node[circle, fill, scale=0.5] (0) at (0, 0) {}; 
	\node[circle, fill, scale=0.5] (1) at (0, 2) {}; 
	\node[circle, fill, scale=0.5] (2) at (1, 1) {}; 
	\node[circle, fill, scale=0.5] (3) at (5, 0) {}; 
	\node[circle, fill, scale=0.5] (4) at (5, 2) {}; 
	\node[circle, fill, scale=0.5] (5) at (4, 1) {}; 
	\node[circle, fill, scale=0.5] (6) at (1.6, 2) {}; 
	\node[circle, fill, scale=0.5] (7) at (1.6, 1) {}; 
	\node[circle, fill, scale=0.5] (8) at (2.2, 1) {}; 
	\node[circle, fill, scale=0.5] (9) at (2.2, 0) {}; 
	\node[circle, fill, scale=0.5] (10) at (2.8, 1) {}; 
	\node[circle, fill, scale=0.5] (11) at (2.8, 0) {}; 
	\node[circle, fill, scale=0.5] (12) at (3.4, 2) {}; 
	\node[circle, fill, scale=0.5] (13) at (3.4, 1) {}; 

    \node at (2.53,0.5) {$\cdots$};

    \draw[white] (1) circle[radius=0.2];
    \draw[white] (3) circle[radius=0.2];

	\draw[black] (0) to (1);
	\draw[black, dashed] (0) to (2);
	\draw[black] (0) to (9);
	\draw[black] (1) to (2);
	\draw[black] (1) to (6);
	\draw[black] (2) to (7);
	\draw[black] (3) to (4);
	\draw[black, dashed] (3) to (5);
	\draw[black] (3) to (11);
	\draw[black] (4) to (5);
	\draw[black] (4) to (12);
	\draw[black] (5) to (13);
	\draw[black] (6) to (7);
	\draw[black] (6) to (12);
	\draw[black] (7) to (8);
	\draw[black] (8) to (9);
	\draw[black] (8) to (10);
	\draw[black] (9) to (11);
	\draw[black] (10) to (11);
	\draw[black] (10) to (13);
	\draw[black] (12) to (13);

    \draw[brace] ($ (11) + (.25,0) $)-- node[below=2ex]{even} ($ (9) + (-.25,0) $);

\end{tikzpicture}
            }
            \label{pat:31..13_even}
            
            \item \parbox{\linewidth}{
                \begin{tikzpicture}[scale=1, brace/.style={thick,decorate, decoration={calligraphic brace, amplitude=7pt,raise=0.5ex}}]
	\node[circle, fill, scale=0.5] (0) at (0, 0) {}; 
	\node[circle, fill, scale=0.5] (1) at (0, 2) {}; 
	\node[circle, fill, scale=0.5] (2) at (1, 1) {}; 
	\node[circle, fill, scale=0.5] (3) at (5, 0) {}; 
	\node[circle, fill, scale=0.5] (4) at (5, 2) {}; 
	\node[circle, fill, scale=0.5] (5) at (4, 1) {}; 
	\node[circle, fill, scale=0.5] (6) at (1.6, 1) {}; 
	\node[circle, fill, scale=0.5] (7) at (1.6, 0) {}; 
	\node[circle, fill, scale=0.5] (8) at (2.2, 1) {}; 
	\node[circle, fill, scale=0.5] (9) at (2.2, 0) {}; 
	\node[circle, fill, scale=0.5] (10) at (2.8, 1) {}; 
	\node[circle, fill, scale=0.5] (11) at (2.8, 0) {}; 
	\node[circle, fill, scale=0.5] (12) at (3.4, 2) {}; 
	\node[circle, fill, scale=0.5] (13) at (3.4, 1) {}; 

    \node at (2.53,0.5) {$\cdots$};

    \draw[white] (1) circle[radius=0.2];
    \draw[white] (3) circle[radius=0.2];

	\draw[black, line width=\w] (0.center) to (1.center);
	\draw[black, line width=\w] (0.center) to (2.center);
	\draw[black, line width=\w] (0.center) to (7.center);
	\draw[black, dashed] (1.center) to (2.center);
	\draw[black] (1) to (12);
	\draw[black] (2) to (6);
	\draw[black] (3) to (4);
	\draw[black, line width=\w, dashed] (3.center) to (5.center);
	\draw[black] (3) to (11);
	\draw[black] (4) to (5);
	\draw[black, line width=\w] (4.center) to (12.center);
	\draw[black] (5) to (13);
	\draw[black] (6) to (7);
	\draw[black, line width=\w] (6.center) to (8.center);
	\draw[black] (7) to (9);
	\draw[black] (8) to (9);
	\draw[black] (8) to (10);
	\draw[black, line width=\w] (9.center) to (11.center);
	\draw[black] (10) to (11);
	\draw[black, line width=\w] (10.center) to (13.center);
	\draw[black, dashed] (12) to (13);

     \draw[thick] (0) circle[radius=0.2];
     \draw[brace] ($ (11) + (.25,0) $)-- node[below=2ex]{odd} ($ (7) + (-.25,0) $);
 
\end{tikzpicture}
                \Arrow
                \begin{tikzpicture}[scale=1, brace/.style={thick,decorate, decoration={calligraphic brace, amplitude=7pt,raise=0.5ex}}]
	\node[circle, fill, scale=0.5] (0) at (0, 0) {}; 
	\node[circle, fill, scale=0.5] (1) at (0, 2) {}; 
	\node[circle, fill, scale=0.5] (2) at (1, 1) {}; 
	\node[circle, fill, scale=0.5] (3) at (5, 0) {}; 
	\node[circle, fill, scale=0.5] (4) at (5, 2) {}; 
	\node[circle, fill, scale=0.5] (5) at (4, 1) {}; 
	\node[circle, fill, scale=0.5] (6) at (1.5, 2) {}; 
	\node[circle, fill, scale=0.5] (7) at (1.5, 1) {}; 
	\node[circle, fill, scale=0.5] (8) at (2.0, 1) {}; 
	\node[circle, fill, scale=0.5] (9) at (2.0, 0) {}; 
	\node[circle, fill, scale=0.5] (10) at (2.5, 1) {}; 
	\node[circle, fill, scale=0.5] (11) at (2.5, 0) {}; 
	\node[circle, fill, scale=0.5] (12) at (3.0, 1) {}; 
	\node[circle, fill, scale=0.5] (13) at (3.0, 0) {}; 
	\node[circle, fill, scale=0.5] (14) at (3.5, 2) {}; 
	\node[circle, fill, scale=0.5] (15) at (3.5, 1) {}; 

    \node at (2.78,0.5) {$\cdots$};

    \draw[white] (1) circle[radius=0.2];
    \draw[white] (3) circle[radius=0.2];

	\draw[black] (0) to (1);
	\draw[black] (0) to (2);
	\draw[black] (0) to (9);
	\draw[black] (1) to (2);
	\draw[black] (1) to (6);
	\draw[black] (2) to (7);
	\draw[black] (3) to (4);
	\draw[black] (3) to (5);
	\draw[black] (3) to (13);
	\draw[black] (4) to (5);
	\draw[black] (4) to (14);
	\draw[black] (5) to (15);
	\draw[black, dashed] (6) to (7);
	\draw[black] (6) to (14);
	\draw[black] (7) to (8);
	\draw[black] (8) to (9);
	\draw[black] (8) to (10);
	\draw[black] (9) to (11);
	\draw[black] (10) to (11);
	\draw[black] (10) to (12);
	\draw[black] (11) to (13);
	\draw[black] (12) to (13);
	\draw[black] (12) to (15);
	\draw[black, dashed] (14) to (15);

    \draw[brace] ($ (13) + (.25,0) $)-- node[below=2ex]{odd} ($ (9) + (-.25,0) $);

\end{tikzpicture}
            }
            \label{pat:31..13_odd}
    
            \item \parbox{\linewidth}{
                \begin{tikzpicture}[scale=1, brace/.style={thick,decorate, decoration={calligraphic brace, amplitude=7pt,raise=0.5ex}}]
	\node[circle, fill, scale=0.5] (0) at (0, 0) {}; 
	\node[circle, fill, scale=0.5] (1) at (0, 2) {}; 
	\node[circle, fill, scale=0.5] (2) at (1, 1) {}; 
	\node[circle, fill, scale=0.5] (3) at (5, 0) {}; 
	\node[circle, fill, scale=0.5] (4) at (5, 2) {}; 
	\node[circle, fill, scale=0.5] (5) at (4, 1) {}; 
	\node[circle, fill, scale=0.5] (6) at (1.75, 1) {}; 
	\node[circle, fill, scale=0.5] (7) at (1.75, 0) {}; 
	\node[circle, fill, scale=0.5] (8) at (2.5, 1) {}; 
	\node[circle, fill, scale=0.5] (9) at (2.5, 0) {}; 
	\node[circle, fill, scale=0.5] (10) at (3.25, 2) {}; 
	\node[circle, fill, scale=0.5] (11) at (3.25, 1) {}; 
    \node at (2.15,0.5) {$\cdots$};

    \draw[white] (1) circle[radius=0.2];
    \draw[white] (3) circle[radius=0.2];

	\draw[black, dashed] (0) to (1);
	\draw[black, line width=\w] (0.center) to (2.center);
	\draw[black] (0) to (7);
	\draw[black, line width=\w] (1.center) to (2.center);
	\draw[black] (1) to (10);
	\draw[black, line width=\w] (2.center) to (6.center);
	\draw[black] (3) to (4);
	\draw[black, line width=\w, dashed] (3.center) to (5.center);
	\draw[black] (3) to (9);
	\draw[black] (4) to (5);
	\draw[black, line width=\w] (4.center) to (10.center);
	\draw[black] (5) to (11);
	\draw[black] (6) to (7);
	\draw[black] (6) to (8);
	\draw[black, line width=\w] (7.center) to (9.center);
	\draw[black] (8) to (9);
	\draw[black, line width=\w] (8.center) to (11.center);
	\draw[black, dashed] (10) to (11);

     \draw[thick] (2) circle[radius=0.2];
     \draw[brace] ($ (9) + (.25,0) $)-- node[below=2ex]{even} ($ (7) + (-.25,0) $);
     

\end{tikzpicture}
                \Arrow
                \begin{tikzpicture}[scale=1, brace/.style={thick,decorate, decoration={calligraphic brace, amplitude=7pt,raise=0.5ex}}]
	\node[circle, fill, scale=0.5] (0) at (0, 0) {}; 
	\node[circle, fill, scale=0.5] (1) at (0, 2) {}; 
	\node[circle, fill, scale=0.5] (2) at (1, 1) {}; 
	\node[circle, fill, scale=0.5] (3) at (5, 0) {}; 
	\node[circle, fill, scale=0.5] (4) at (5, 2) {}; 
	\node[circle, fill, scale=0.5] (5) at (4, 1) {}; 
	\node[circle, fill, scale=0.5] (6) at (1.6, 2) {}; 
	\node[circle, fill, scale=0.5] (7) at (1.6, 0) {}; 
	\node[circle, fill, scale=0.5] (8) at (2.2, 1) {}; 
	\node[circle, fill, scale=0.5] (9) at (2.2, 0) {}; 
	\node[circle, fill, scale=0.5] (10) at (2.8, 1) {}; 
	\node[circle, fill, scale=0.5] (11) at (2.8, 0) {}; 
	\node[circle, fill, scale=0.5] (12) at (3.4, 2) {}; 
	\node[circle, fill, scale=0.5] (13) at (3.4, 1) {}; 

    \node at (2.53,0.5) {$\cdots$};

    \draw[white] (1) circle[radius=0.2];
    \draw[white] (3) circle[radius=0.2];

	\draw[black, dashed] (0) to (1);
	\draw[black] (0) to (2);
	\draw[black] (0) to (7);
	\draw[black] (1) to (2);
	\draw[black] (1) to (6);
	\draw[black] (2) to (8);
	\draw[black] (3) to (4);
	\draw[black] (3) to (5);
	\draw[black] (3) to (11);
	\draw[black] (4) to (5);
	\draw[black] (4) to (12);
	\draw[black] (5) to (13);
	\draw[black] (6) to (7);
	\draw[black] (6) to (12);
	\draw[black] (7) to (9);
	\draw[black] (8) to (9);
	\draw[black] (8) to (10);
	\draw[black] (9) to (11);
	\draw[black] (10) to (11);
	\draw[black] (10) to (13);
	\draw[black, dashed] (12) to (13);

    \draw[brace] ($ (11) + (.25,0) $)-- node[below=2ex]{even} ($ (9) + (-.25,0) $);

\end{tikzpicture}
            }
            \label{pat:21..13_even}
    
            \item \parbox{\linewidth}{
                \begin{tikzpicture}[scale=1, brace/.style={thick,decorate, decoration={calligraphic brace, amplitude=7pt,raise=0.5ex}}]
	\node[circle, fill, scale=0.5] (0) at (0, 0) {}; 
	\node[circle, fill, scale=0.5] (1) at (0, 2) {}; 
	\node[circle, fill, scale=0.5] (2) at (1, 1) {}; 
	\node[circle, fill, scale=0.5] (3) at (5, 0) {}; 
	\node[circle, fill, scale=0.5] (4) at (5, 2) {}; 
	\node[circle, fill, scale=0.5] (5) at (4, 1) {}; 
	\node[circle, fill, scale=0.5] (6) at (1.6, 1) {}; 
	\node[circle, fill, scale=0.5] (7) at (1.6, 0) {}; 
	\node[circle, fill, scale=0.5] (8) at (2.2, 1) {}; 
	\node[circle, fill, scale=0.5] (9) at (2.2, 0) {}; 
	\node[circle, fill, scale=0.5] (10) at (2.8, 1) {}; 
	\node[circle, fill, scale=0.5] (11) at (2.8, 0) {}; 
	\node[circle, fill, scale=0.5] (12) at (3.4, 2) {}; 
	\node[circle, fill, scale=0.5] (13) at (3.4, 1) {}; 

    \node at (2.53,0.5) {$\cdots$};

    \draw[white] (1) circle[radius=0.2];
    \draw[white] (3) circle[radius=0.2];

	\draw[black] (0) to (1);
	\draw[black, line width=\w] (0.center) to (2.center);
	\draw[black] (0) to (7);
	\draw[black, line width=\w, dashed] (1.center) to (2.center);
	\draw[black] (1) to (12);
	\draw[black, line width=\w] (2.center) to (6.center);
	\draw[black] (3) to (4);
	\draw[black, dashed] (3) to (5);
	\draw[black, line width=\w] (3.center) to (11.center);
	\draw[black, line width=\w] (4.center) to (5.center);
	\draw[black] (4) to (12);
	\draw[black] (5) to (13);
	\draw[black] (6) to (7);
	\draw[black] (6) to (8);
	\draw[black, line width=\w] (7.center) to (9.center);
	\draw[black] (8) to (9);
	\draw[black, line width=\w] (8.center) to (10.center);
	\draw[black] (9) to (11);
	\draw[black] (10) to (11);
	\draw[black] (10) to (13);
	\draw[black, line width=\w, dashed] (12.center) to (13.center);

     \draw[thick] (2) circle[radius=0.2];
     \draw[brace] ($ (11) + (.25,0) $)-- node[below=2ex]{odd} ($ (7) + (-.25,0) $);

 
\end{tikzpicture}
                \Arrow
                \begin{tikzpicture}[scale=1, brace/.style={thick,decorate, decoration={calligraphic brace, amplitude=7pt,raise=0.5ex}}]
	\node[circle, fill, scale=0.5] (0) at (0, 0) {}; 
	\node[circle, fill, scale=0.5] (1) at (0, 2) {}; 
	\node[circle, fill, scale=0.5] (2) at (1, 1) {}; 
	\node[circle, fill, scale=0.5] (3) at (5, 0) {}; 
	\node[circle, fill, scale=0.5] (4) at (5, 2) {}; 
	\node[circle, fill, scale=0.5] (5) at (4, 1) {}; 
	\node[circle, fill, scale=0.5] (6) at (1.5, 2) {}; 
	\node[circle, fill, scale=0.5] (7) at (1.5, 0) {}; 
	\node[circle, fill, scale=0.5] (8) at (2.0, 1) {}; 
	\node[circle, fill, scale=0.5] (9) at (2.0, 0) {}; 
	\node[circle, fill, scale=0.5] (10) at (2.5, 1) {}; 
	\node[circle, fill, scale=0.5] (11) at (2.5, 0) {}; 
	\node[circle, fill, scale=0.5] (12) at (3.0, 1) {}; 
	\node[circle, fill, scale=0.5] (13) at (3.0, 0) {}; 
	\node[circle, fill, scale=0.5] (14) at (3.5, 2) {}; 
	\node[circle, fill, scale=0.5] (15) at (3.5, 1) {}; 

    \node at (2.78,0.5) {$\cdots$};

    \draw[white] (1) circle[radius=0.2];
    \draw[white] (3) circle[radius=0.2];

	\draw[black] (0) to (1);
	\draw[black, dashed] (0) to (2);
	\draw[black] (0) to (7);
	\draw[black] (1) to (2);
	\draw[black] (1) to (6);
	\draw[black] (2) to (8);
	\draw[black] (3) to (4);
	\draw[black, dashed] (3) to (5);
	\draw[black] (3) to (13);
	\draw[black] (4) to (5);
	\draw[black] (4) to (14);
	\draw[black] (5) to (15);
	\draw[black] (6) to (7);
	\draw[black] (6) to (14);
	\draw[black] (7) to (9);
	\draw[black] (8) to (9);
	\draw[black] (8) to (10);
	\draw[black] (9) to (11);
	\draw[black] (10) to (11);
	\draw[black] (10) to (12);
	\draw[black] (11) to (13);
	\draw[black] (12) to (13);
	\draw[black] (12) to (15);
	\draw[black] (14) to (15);

    \draw[brace] ($ (13) + (.25,0) $)-- node[below=2ex]{odd} ($ (9) + (-.25,0) $);

\end{tikzpicture}
            }
            \label{pat:21..13_odd}
        \end{enumerate}
        \caption{On the right-hand side of the picture we show extended prisms with at most two lonely edges (dashed). They can be obtained by replacing the encircled vertex $v$ of the graph on the left-hand side with a triangle, where the thick edges represent joins rooted in $v$.}
        \label{fig:bad_patterns_2}
    \end{figure}

    In Figure~\ref{fig:bad_patterns} and~\ref{fig:bad_patterns_2}, we depict (on the right-hand side) a list of extended prisms having at most two lonely edges. Indeed, from the previous characterization of $\mathcal{U}_4$ and Lemma \ref{lem:blowUpToTriangle} it is easy to see that the graphs on the left-hand side of such figures have exactly $3$ lonely edges. Applying again Lemma \ref{lem:blowUpToTriangle} we get that the graphs on the right-hand side have exactly two lonely edges.
    By Observation~\ref{obs:subsequence}, if one of their extension patterns appears as a subsequence of a pattern of an extended prism $P'$, we can deduce that $P'$ is in $\mathcal{U}_{\le 2}$, where $\mathcal{U}_{\le k} := \bigcup_{i=0}^k\mathcal{U}_i$.

    So let $Y=(s_1,s_2,s_3,.....,s_k)$ with $k\ge4$ be an extension pattern of $G$, and let $\operatorname{set}(Y)$ be the set of elements appearing in the sequence $Y$. 
    If $|\operatorname{set}(Y)|=1$, by Theorem~\ref{thm:Characterization_U4}, we know that $G\in\mathcal{U}_4$. If $|\operatorname{set}(Y)|=2$, we have the following cases:
    \begin{itemize}
        \item there is a $j$ such that $s_j = s_{j+1} \ne s_{j+2} = s_{j+3}$. This case leads to a graph in $\mathcal{U}_{\le 2}$ 
        because of pattern~\ref{pat:1133} in Figure~\ref{fig:bad_patterns};
        \item there are $i<j$ with $j-i \ge 2$ such that, $s_i = s_j$ and such that $s_i \ne s_{i+1}=s_{i+2}=...=s_{j-1} \ne s_j$. This case leads to a graph in $\mathcal{U}_{\le 2}$ because of patterns \ref{pat:31..13_even} and \ref{pat:31..13_odd} in Figure~\ref{fig:bad_patterns_2}.
    \end{itemize}
    So we conclude that, in this case, the only pattern left (up to isomorphism of graphs) is $s_1 = s_2 = \ldots = s_{k-1} \neq s_k$. 
    This is indeed the only possible pattern (up to graph isomorphism) allowed by the statement when $|\operatorname{set}(Y)|=2$.

    If $|$set$(Y)|=3$ we have the following cases:
    \begin{itemize}
        \item all patterns containing a pattern described in the previous two points lead to graphs in $\mathcal{U}_{\le 2}$;
        \item there are $i<j$ such that $j-i \ge 3$ such that $s_i \ne s_j$ and $s_i \ne s_{i+1}=s_{i+2}=...=s_{j-1} \ne s_j.$ This case leads to a graph in $\mathcal{U}_{\le 2}$ because of patterns~\ref{pat:21..13_even} and~\ref{pat:21..13_odd} in Figure~\ref{fig:bad_patterns_2};
        \item there is an integer $j$ such that $s_j = s_{j+3}$ and $s_j\ne s_{j+1} \ne s_{j+2} \ne s_{j+3}$. This case leads to a graph in $\mathcal{U}_{\le 2}$ because of pattern~\ref{pat:2132} in Figure~\ref{fig:bad_patterns}.
    \end{itemize}
    So we conclude that also in this case the only patterns left are the ones described in our statement.

Hence, up to now we have proved that if $G \in \mathcal{U}_3$ then either it is in $\mathcal{W}_3$ or it is an extended prism with extension pattern $(1,\ldots,1,3,2,\ldots,2)$.\\

Now we conclude the proof by showing the converse, that is: all patterns in the statement give rise to extended prisms with exactly three lonely edges. Every extended prism with a pattern as stated in the theorem is a descendant of  $Pr_3$ or $Pr_4$, and then by Observation \ref{obs:subsequence} it has at most three lonely edges. 
Hence, we only need to prove that it has at least three lonely edges to conclude that it belongs to $\mathcal{U}_3.$

We prove the following: if $G$ is an extended prism admitting pattern $( \underbrace{1,\dots,1}_{h-1 \text{ times}}, 3, \underbrace{2,\dots,2}_{k-h \text{ times}}), $ where $h-1 \geq k-h$, $ 2 \leq h \leq k$ and $k\ge 3$, then 
    
    \begin{itemize}
        \item[$i)$] the only edge $e'$ (see Figure \ref{fig:u3join}) corresponding to the $T_v$-extension of index $3$ is lonely;
        \item[$ii)$] if $h$ is even then 
        $u_1u_2$ is lonely, otherwise 
        $u_1u_3$ is lonely;
        \item[$iii)$] if $k-h$ is odd then 
        $v_1v_2$ is lonely, otherwise 
        $v_2v_3$ is lonely.
    \end{itemize}
    
    For $k\in\{h,h+1\},$ we prove this statement by induction on $h$. The base step is given by the graphs $Pr_3$ and $Pr_4$ defined above in this proof. More precisely, for $k=h$ the base step is $Pr_3$, whereas for $k=h+1$, the base step is given by $Pr_4$. 
    So, let $h\ge3$ and $G$ be an extended prism with pattern described above.  
    Note that $G$ is obtained by replacing the vertex 
    $u_1$ of the graph $G'$ by a triangle, where $G'$ is an extended prism with pattern $(\underbrace{1,\dots,1}_{h-2 \text{ times}}, 3, \underbrace{2,\dots,2}_{k-h \text{ times}})$. Points $i), ii)$ and $iii)$ hold for $G'$ by inductive hypothesis.
    When constructing $G$ from $G'$, it is easy to see that the lonely edge in $T_u$ changes according to the parity of $h$ as stated above. 
    Now, we just show that there is no 
    join rooted in $u_1$ containing lonely edges in $G'$.
    \begin{figure}[!htb]
        \centering
        \begin{tikzpicture}[scale=1]
	\node[circle, fill, scale=0.5, label=below:$u_1$] (0) at (0, 0) {}; 
	\node[circle, fill, scale=0.5, label=$u_2$] (1) at (0, 2) {}; 
	\node[circle, fill, scale=0.5, label=$u_3$] (2) at (1, 1) {}; 
	\node[circle, fill, scale=0.5, label=below:$v_1$] (3) at (5, 0) {}; 
	\node[circle, fill, scale=0.5, label=$v_2$] (4) at (5, 2) {}; 
	\node[circle, fill, scale=0.5, label=$v_3$] (5) at (4, 1) {}; 
	\node[circle, fill, scale=0.5] (6) at (1.75, 2) {}; 
	\node[circle, fill, scale=0.5] (7) at (1.75, 1) {}; 
	\node[circle, fill, scale=0.5] (8) at (2.5, 2) {}; 
	\node[circle, fill, scale=0.5] (9) at (2.5, 1) {}; 
	\node[circle, fill, scale=0.5] (10) at (3.25, 2) {}; 
	\node[circle, fill, scale=0.5] (11) at (3.25, 0) {}; 

	\draw[black, line width=\w] (0.center) to (1.center);
	\draw[black, line width=\w, dashed] (0.center) to (2.center);
	\draw[black, line width=\w] (0.center) to (11.center);
	\draw[black] (1) to (2);
	\draw[black] (1) to (6);
	\draw[black] (2) to (7);
	\draw[black] (3) to (4);
	\draw[black] (3) to (5);
	\draw[black] (3) to (11);
	\draw[black, line width=\w, dashed] (4.center) to (5.center);
	\draw[black] (4) to (10);
	\draw[black] (5) to (9);
	\draw[black] (6) to (7);
	\draw[black] (6) to (8);
	\draw[black] (7) to (9);
	\draw[black] (8) to (9);
	\draw[black] (8) to (10);
	\draw[black, dashed] (10) to node [near end, left] {$e'$} (11);
 
\end{tikzpicture}\qquad
        \begin{tikzpicture}[scale=1]
	\node[circle, fill, scale=0.5, label=below:$u_1$] (0) at (0, 0) {}; 
	\node[circle, fill, scale=0.5, label=$u_2$] (1) at (0, 2) {}; 
	\node[circle, fill, scale=0.5, label=$u_3$] (2) at (1, 1) {}; 
	\node[circle, fill, scale=0.5, label=below:$v_1$] (3) at (5, 0) {}; 
	\node[circle, fill, scale=0.5, label=$v_2$] (4) at (5, 2) {}; 
	\node[circle, fill, scale=0.5, label=$v_3$] (5) at (4, 1) {}; 
	\node[circle, fill, scale=0.5] (6) at (1.6, 2) {}; 
	\node[circle, fill, scale=0.5] (7) at (1.6, 1) {}; 
	\node[circle, fill, scale=0.5] (8) at (2.2, 2) {}; 
	\node[circle, fill, scale=0.5] (9) at (2.2, 1) {}; 
	\node[circle, fill, scale=0.5] (10) at (2.8, 2) {}; 
	\node[circle, fill, scale=0.5] (11) at (2.8, 0) {}; 
	\node[circle, fill, scale=0.5] (12) at (3.4, 1) {}; 
	\node[circle, fill, scale=0.5] (13) at (3.4, 0) {}; 

	\draw[black, line width=\w] (0.center) to (1.center);
	\draw[black, line width=\w, dashed] (0.center) to (2.center);
	\draw[black, line width=\w] (0.center) to (11.center);
	\draw[black] (1) to (2);
	\draw[black] (1) to (6);
	\draw[black] (2) to (7);
	\draw[black, line width=\w, dashed] (3.center) to (4.center);
	\draw[black] (3) to (5);
	\draw[black] (3) to (13);
	\draw[black] (4) to (5);
	\draw[black] (4) to (10);
	\draw[black, line width=\w] (5.center) to (12.center);
	\draw[black] (6) to (7);
	\draw[black] (6) to (8);
	\draw[black] (7) to (9);
	\draw[black] (8) to (9);
	\draw[black] (8) to (10);
	\draw[black] (9) to (12);
	\draw[black, dashed] (10) to node [near start, right] {$e'$} (11);
	\draw[black] (11) to (13);
	\draw[black] (12) to (13);
\end{tikzpicture}
        \caption{The graph $G'$ with $h = 3$. On the left $k = 3$ and on the right $k = 4$. The thick edges indicate edges which necessarily need to be present in a 
        join rooted in $u_1$ containing a lonely edge not incident to $u_3$. }
        \label{fig:u3join}
    \end{figure}
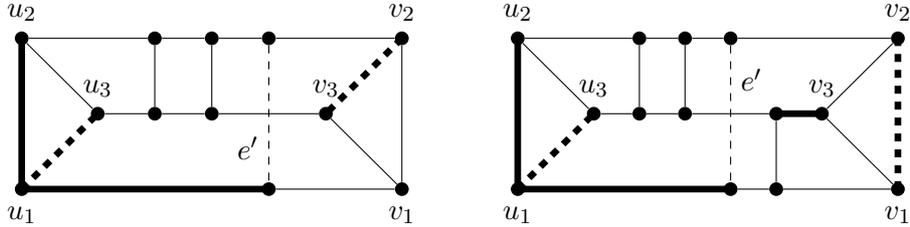
    
    Note that a 
    join rooted in $u_1$ cannot contain $e'$ and it cannot contain the other lonely edge. See Figure~\ref{fig:u3join}, where we depict $G'$ with $h=3$ and $k\in \{3,4\}$. Indeed, a 
    join rooted in $u_1$ contains an edge incident with $e'$. If it does contain the other lonely edge then one vertex is unmatched. Note that this holds for all values of $h$. Thus the statement follows for all $h\ge2$ and $k\in\{h,h+1\}$ with $k\ge3.$

    The same argument now can be employed to prove the statement for all other cases. Namely, we fix $h\ge2$ and we argue by induction on $k$. 
\end{proof}

\subsection{Some insights on \texorpdfstring{$\mathcal{U}_1$}{U1} and \texorpdfstring{$\mathcal{U}_2$}{U2}}

We conclude the paper with some insights on ${\cal U}_1$ and ${\cal U}_2$, which turn out to be much richer families of graphs than ${\cal U}_k$ for $k>2.$
Indeed, we prove something stronger, i.e.\ that for all $k\ge2$ there is a graph in $\mathcal{U}_1$ and a graph in $\mathcal{U}_2$ with exactly $k$ triangles.
Note that, except for $K_4$, every graph in 
$\mathcal{U} \setminus \mathcal{U}_{\le 2}$ 
contains at most three triangles. In particular, finitely many such graphs contain exactly three triangles and all others contain exactly two triangles. 


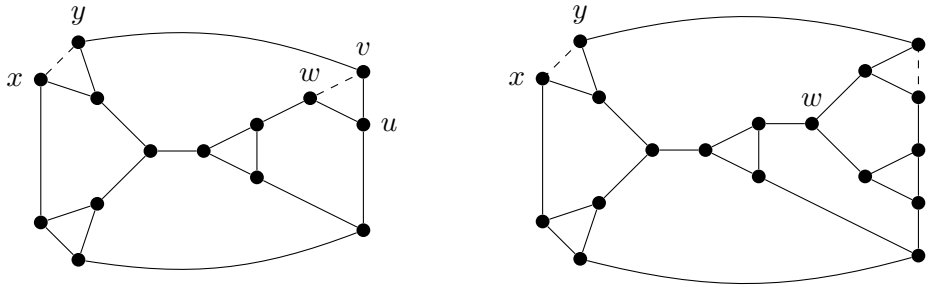
\begin{figure}[!htb]
    \centering
        







    \begin{tikzpicture}[scale=0.7]
        \node[circle, fill, scale=0.5] (1) at (-0.35355339059, -2.06066017178) {};
        \node[circle, fill, scale=0.5] (2) at (-1.06066017178, -1.35355339059) {};
        \node[circle, fill, scale=0.5] (3) at (0, -1) {};
        
        \node[circle, fill, scale=0.5, label=left:$x$] (4) at (-1.06066017178, 1.35355339059) {};
        \node[circle, fill, scale=0.5, label=$y$] (5) at (-0.35355339059, 2.06066017178) {};
        \node[circle, fill, scale=0.5] (6) at (0, 1) {};
        \node[circle, fill, scale=0.5] (7) at (1, 0) {};

        \draw[] (1) to (2);
        \draw (1) to (3);
        \draw (2) to (3);

        \draw (2) to[] (4);

        \draw[dashed] (4) to (5);
        \draw (4) to (6);
        \draw (5) to (6);

        \draw (3) to (7);
        \draw (6) to (7);


        \node[circle, fill, scale=0.5] (8) at (2, 0) {};
        \node[circle, fill, scale=0.5] (9) at (3, -0.5) {};
        \node[circle, fill, scale=0.5] (10) at (3, 0.5) {};

        \draw (8) to (9);
        \draw (8) to (10);
        \draw (9) to (10);

        \draw (7) to (8);

        \node[circle, fill, scale=0.5, label=$w$] (11) at (4, 1) {};
        \node[circle, fill, scale=0.5, label=$v$] (12) at (5, 1.5) {};
        \node[circle, fill, scale=0.5, label= right:$u$] (13) at (5, 0.5) {};

        \draw[dashed] (11) to (12);
        \draw (11) to (13);
        \draw[] (12) to (13);

        \draw (10) to (11);

        \node[circle, fill, scale=0.5] (14) at (5, -1.5) {};
        \draw (9) to (14);
        \draw (13) to (14);

        \draw (1) to[bend right=15]  (14);
        \draw (5) to[bend left=15] (12);

        \node[circle, fill=white] at (6, 2.5) {};
        \node[fill=white] at (6, -2.5) {};
    \end{tikzpicture}
    \qquad 
    \begin{tikzpicture}[scale=0.7]
        \node[circle, fill, scale=0.5] (1) at (-0.35355339059, -2.06066017178) {};
        \node[circle, fill, scale=0.5] (2) at (-1.06066017178, -1.35355339059) {};
        \node[circle, fill, scale=0.5] (3) at (0, -1) {};
        
        \node[circle, fill, scale=0.5, label=left:$x$] (4) at (-1.06066017178, 1.35355339059) {};
        \node[circle, fill, scale=0.5, label=above:$y$] (5) at (-0.35355339059, 2.06066017178) {};
        \node[circle, fill, scale=0.5] (6) at (0, 1) {};
        \node[circle, fill, scale=0.5] (7) at (1, 0) {};

        \draw[] (1) to (2);
        \draw (1) to (3);
        \draw (2) to (3);

        \draw (2) to[] (4);

        \draw[dashed] (4.center) to (5.center);
        \draw (4) to (6);
        \draw (5) to (6);

        \draw (3) to (7);
        \draw (6) to (7);

        \node[circle, fill, scale=0.5] (8) at (2, 0) {};
        \node[circle, fill, scale=0.5] (9) at (3, -0.5) {};
        \node[circle, fill, scale=0.5] (10) at (3, 0.5) {};

        \draw (8) to (9);
        \draw (8) to node[above left] {} (10);
        \draw (9) to (10);

        \draw (7) to (8);

        \node[circle, fill, scale=0.5] (11) at (6, -2) {};
        \draw (9) to (11);

        \node[circle, fill, scale=0.5, label=above:$w$] (12) at (4, 0.5) {};
        \draw (10) to (12);

        \node[circle, fill, scale=0.5] (13) at (5, 1.5) {};
        \node[circle, fill, scale=0.5] (14) at (6, 2) {};
        \node[circle, fill, scale=0.5, label= {[label distance=-0.15cm]above left:}] (15) at (6, 1) {};
        \draw (13) to (14);
        \draw (13) to (15);
        \draw[dashed] (14.center) to (15.center);

        \node[circle, fill, scale=0.5] (16) at (5, -0.5) {};
        \node[circle, fill, scale=0.5] (17) at (6, -1) {};
        \node[circle, fill, scale=0.5] (18) at (6, 0) {};

        \draw (16) to (17);
        \draw[] (16.center) to (18.center);
        \draw[] (17) to (18);

        \draw (11) to (17);
        \draw (12.center) to (13.center);
        \draw (12) to (16);
        \draw (15) to (18);

        \draw (1) to[bend right=15]  (11);
        \draw (5) to[bend left=15] (14);
    \end{tikzpicture}
    \caption{
    {Left-hand side:} 
    The graph $G_4$ with (dashed) lonely edges $xy$ and $vw$.
    {Right-hand side:} 
    The graph $G_5 := (G_4^u)^v$. Its lonely edges are dashed.}
    \label{fig:gadget_and_G0}
\end{figure}

\begin{theorem}\label{thm:U1_U2}
    For all $k \ge 2$, there is a graph in $\mathcal{U}_2$ and a graph in $\mathcal{U}_1$ with exactly $k$ triangles.
\end{theorem}

\begin{proof}
We first prove the statement for $\mathcal{U}_2$.

For $k\in\{2,3\}$ there is a graph in $\mathcal{U}_2$ with exactly $k$ triangles, see for example Figures~\ref{fig:bad_patterns},~\ref{fig:bad_patterns_2} and~\ref{fig:2lonelyEdges}. We prove that for every $k\geq 4$ there is a graph in $\mathcal{U}_2$ with exactly $k$ triangles that has the following additional properties:
\begin{enumerate}[$i)$]
    \item the two lonely edges are at distance 2,
    \item both lonely edges belong to a triangle, and
    \item there is no perfect matching containing both lonely edges.
\end{enumerate}
We prove the statement above by induction on $k$. For $k=4$ an example is given in Figure~\ref{fig:gadget_and_G0}. Hence, let $G_k$ be a graph in $\mathcal{U}_2$ with exactly $k$ triangles that satisfies $i)$, $ii)$ and $iii)$. We use $G_k$ to construct a graph $G_{k+1}$ with exactly $k+1$ triangles that belongs to $\mathcal{U}_2$ and also satisfies $i)$, $ii)$ and $iii)$. Let $xy$ and $vw$ be the lonely edges of $G_k$, where $y$ and $v$ are adjacent and let $u$ be the third vertex belonging to the triangle that contains $vw$ (see Figure~\ref{fig:gadget_and_G0}).
Set $G_{k+1}=(G^u)^v$. We claim that $G_{k+1}$ has the desired properties. By construction, $G_{k+1}$ has exactly $k+1$ triangles. Furthermore, if there is a join $J$ rooted in $u$ in $G_k$ containing $xy$, then replacing the edges $uv, uw$ by the edge $vw$ in $J$ gives a perfect matching containing both lonely edges of $G_k$, a contradiction to $iii)$. Thus, by Lemma~\ref{lem:blowUpToTriangle}, $G^u$ has exactly two lonely edges, namely $xy$ and $vw$. As a consequence, again by Lemma~\ref{lem:blowUpToTriangle}, $G_{k+1}$ has exactly two lonely edges, namely $xy$ and the new edge opposite to $vw$ that is added when replacing $v$ (see Figure~\ref{fig:gadget_and_G0} for an example).
Thus, $G_{k+1}$ belongs to $\mathcal{U}_2$ and satisfies $i)$ and $ii)$. Finally, suppose $G_{k+1}$ has a perfect matching $M$ containing both lonely edges of $G_{k+1}$. By construction, $M$ contains the thick edges depicted in Figure \ref{fig:perfect matching in G_k+1}. Thus, we can transform $M$ into a perfect matching in $G_k$ containing $xy$ and $vw$, a contradiction since $G_{k}$ satisfies $iii)$.

\begin{figure}[!htb]
    \centering
    \begin{tikzpicture}
        \node[circle, fill, scale=0.5] (1) at (-1.5, -1) {};
        \node[circle, fill, scale=0.5] (2) at (-0.5, -1) {};
        \node[circle, fill, scale=0.5] (3) at (-1, 0) {};
        
        \node[circle, fill, scale=0.5, label=above left:] (4) at (0.5, -1) {};
        \node[circle, fill, scale=0.5, label=above right:] (5) at (1.5, -1) {};
        \node[circle, fill, scale=0.5, label=] (6) at (1, 0) {};

        \node[circle, fill, scale=0.5, label=right:$w$] (7) at (0, 1) {};

        \node[circle, fill, scale=0.5, label= {[label distance=-0.1cm]above left:$y$}] (8) at (2.5, -1) {};
        \node[circle, fill, scale=0.5, label={[label distance=-0.1cm]above right:$x$}] (9) at (3.5, -1) {};
        
        \node[circle, fill, scale=0.5, label=right:] (10) at (3, 0) {};

        \draw (1) to (2);
        \draw (1) to (3);
        \draw[line width=\w] (2.center) to (3.center);

        \draw (2) to (4);

        \draw[line width=\w, dashed] (4.center) to (5.center);
        \draw (4) to (6);
        \draw (5) to (6);

        \draw (3) to (7);
        \draw[line width=\w] (6.center) to (7.center);

        \draw[line width=\w] (1.center) to ($ (1) + (-0.5,0) $);
        \draw (7) to ($ (7) + (0,0.5) $);
        
        \draw (5) to (8);
        \draw[line width=\w, dashed] (8.center) to (9.center);
        \draw (9) to ($ (9) + (0.5,0) $);

        \draw (8) to (10);
        \draw (9) to (10);
        \draw[line width=\w] (10.center) to ($ (10) + (0,0.5) $);
        
    \end{tikzpicture}
    \qquad
    \begin{tikzpicture}
        
        \node[circle, fill, scale=0.5, label={[label distance=-0.1cm]above left:$u$}] (4) at (0.5, -1) {};
        \node[circle, fill, scale=0.5, label={[label distance=-0.1cm]above right:$v$}] (5) at (1.5, -1) {};
        \node[circle, fill, scale=0.5, label=right:$w$] (6) at (1, 0) {};


        \node[circle, fill, scale=0.5, label= {[label distance=-0.1cm]above left:$y$}] (8) at (2.5, -1) {};
        \node[circle, fill, scale=0.5, label={[label distance=-0.1cm]above right:$x$}] (9) at (3.5, -1) {};
        
        \node[circle, fill, scale=0.5, label=right:] (10) at (3, 0) {};



        \draw[] (4.center) to (5.center);
        \draw (4) to (6);
        \draw[line width=\w, dashed] (5.center) to (6.center);


        \draw[line width=\w] (4.center) to ($ (4) + (-0.5,0) $);
        \draw[] (6) to ($ (6) + (0,0.5) $);
        
        \draw (5) to (8);
        \draw[line width=\w, dashed] (8.center) to (9.center);
        \draw (9) to ($ (9) + (0.5,0) $);

        \draw (8) to (10);
        \draw (9) to (10);
        \draw[line width=\w] (10.center) to ($ (10) + (0,0.5) $);
        
    \end{tikzpicture}
    
    \caption{
    {Left-hand side:}
    Part of $G_{k+1}$ with $k\geq 4$ and lonely edges dashed. The thick edges are necessary in a perfect matching containing both lonely edges.
    {Right-hand side:}
    A part of $G_k$ with $k\geq 4$ and lonely edges dashed. A perfect matching $M$ containing both lonely edges in $G_{k+1}$ gives rise to a perfect matching in $G_k$ containing the thick edges.}
        
    \label{fig:perfect matching in G_k+1}
\end{figure}
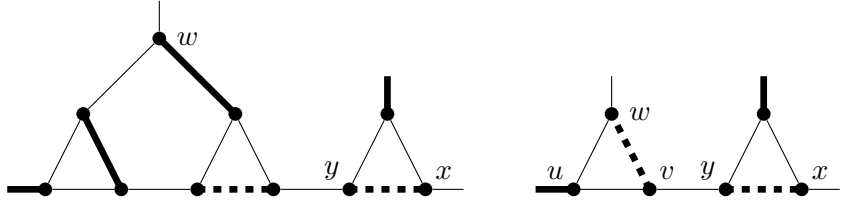

The statement for $\mathcal{U}_1$ can be proved as follows.
For $k\ge4$, it easy to check that in $G_k$ there is a join rooted in $x$ containing the lonely edge $vw$. Therefore, replacing $x$ with a triangle gives rise to a graph in $\mathcal{U}_1$ with exactly $k$ triangles. If $k=3$, we apply the same argument to the graph in Figure \ref{fig:2lonelyEdges}. Finally, we obtain a graph in $\mathcal{U}_1$ with $k=2$ edges by expanding $u_1$ in the second graph on the right-hand side of Figure \ref{fig:bad_patterns}.
\end{proof}

\clearpage
\bibliographystyle{plainurl}
\bibliography{references}
\end{document}